\documentclass[10pt]{amsart}
\pdfoutput=1
\usepackage{amssymb,amscd,graphicx,enumerate,epsfig,color}
\usepackage[bookmarks=true]{hyperref}

\newtheorem{claim}{Claim}[]
\newtheorem{assumption}{Assumption}[]
\newtheorem{theorem}{Theorem}[section]

\newtheorem{lemma}[theorem]{Lemma}

\theoremstyle{definition}
\newtheorem{definition}[theorem]{Definition}

\newcommand{\D}{\mathcal{D}}
\newcommand{\DV}{\mathcal{D}_\mathcal{V}}
\newcommand{\DW}{\mathcal{D}_\mathcal{W}}
\newcommand{\DVW}{\mathcal{D}_{\mathcal{VW}}}

\newenvironment{proofN}[1]{\noindent\textit{Proof of #1.}}{\hfill$\square$\\}

\newcommand{\Case}[1]{\textbf{Case #1.}}

\newcommand{\V}{\mathcal{V}}
\newcommand{\W}{\mathcal{W}}

\begin{document}

\title[]{On the group action of $Mod(M,F)$ on the disk complex}

\author{Jungsoo Kim}
\date{February 23, 2016}

\begin{abstract}
Let $(\V,\W;F)$ be a weakly reducible, unstabilized, Heegaard splitting of genus at least three in an orientable, irreducible $3$-manifold $M$.
Then $Mod(M,F)$ naturally acts on the disk complex $\D(F)$ as a group action.
In this article, we prove if $F$ is topologically minimal and its topological index is two, then the orbit of any element of $\D(F)$ for this group action consists of infinitely many elements.
Moreover, we prove there are at most  two elements of $\D(F)$ whose orbits are finite if the genus of $F$ is three.
\end{abstract}

\address{\parbox{4in}{
	BK21 PLUS SNU Mathematical Sciences Division,\\ Seoul National University\\ 
	1 Gwanak-ro, Gwanak-Gu, Seoul 151-747, Korea\\
}} 
	
\email{pibonazi@gmail.com}

\maketitle
\section{Introduction and Result}
Throughout this paper, all surfaces and $3$-manifolds will be taken to be compact and orientable.
Let $(\V,\W;F)$ be a weakly reducible, unstabilized Heegaard splitting of genus at least three in an orientable, irreducible $3$-manifold $M$.
Let $\D(F)$ be the simplicial complex of isotopy classes of compressing disks of $F$, say the ``\textit{disk complex}'', where an $m$-simplex in $\D(F)$ is defined by mutually disjoint, non-isotopic $(m+1)$-compressing disks of $F$.
In \cite{Bachman2010}, Bachman introduced the concepts a ``\textit{topologically minimal surface}'' and the ``\textit{topological index}'', where we say a separating surface $F$ with no torus components is topologically minimal if either $\D(F)=\emptyset$ or $\D(F)$ is not contractible and the topological index of $F$ refers to the homotopy index of $\D(F)$.
Topologically minimal surfaces include incompressible surfaces, strongly irreducible surfaces as well as certain kinds of weakly reducible surfaces and they share important common properties.
For example, a topologically minimal surface intersects an incompressible surface so that the intersection is essential on both surfaces up to isotopy if the manifold is irreducible.

The group of equivalence classes of automorphisms of $M$ that take $F$ onto itself, say $Mod(M,F)$ (see Section 2 of \cite{JohnsonRubinstein2013}), naturally acts on $\D(F)$, where two automorphisms are \textit{equivalent} if there is an isotopy from one to the other by automorphisms that take $F$ onto itself.
Johnson and Rubinstein proved there is an infinite order element in $Mod(M,F)$ if $F$ is of genus at least three and there are two compressing disks $V\subset \V$ and $W\subset \W$ of $F$ such that $\partial V$ intersects $\partial W$ transversely in exactly two points and $\partial V$ is not isotopic to $\partial W$ in $F$, so called a ``\textit{torus twist}'' (see Section 6 of \cite{JohnsonRubinstein2013} or Section 3 of \cite{Johnson2011}).
Moreover, Johnson gave an exact representation of $Mod(M,F)$ when $M$ is the $3$-torus $T^3$ and $F$ is the standard genus three Heegaard splitting of $T^3$ and he showed the torus twist plays a crucial role in $Mod(M,F)$ as a generator (see Theorem 1 and Lemma 7 of \cite{Johnson2011}).

Motivated by the Johnson and Rubinstein's idea, we will prove the following theorem giving a connection between the topological index of $F$ and the group action of $Mod(M,F)$ on $\D(F)$.

\begin{theorem}[Theorem \ref{theorem-mod-invariant-3-main}]\label{theorem-mod-invariant-3}
Let $(\V,\W;F)$ be a weakly reducible, unstabilized Heegaard splitting of genus at least three in an orientable, irreducible $3$-manifold $M$.
If $F$ is topologically minimal and its topological index is two, then  the orbit $Mod(M,F).[D]$ of any element $[D]\in\D(F)$ consists of infinitely many elements.
This means if there is an element of $\D(F)$ having finite orbit, then (i) $F$ is not topologically minimal or (ii) the topological index of $F$ is at least three if $F$ is topologically minimal.
\end{theorem}

In \cite{JungsooKim2014}, the author proved the topological index of $F$ is two when $F$ is topologically minimal and the genus of $F$ is three and described the shape of the subset of $\D(F)$ consisting of simplices having at least one vertex in $\V$ and at least one vertex in $\W$ when $F$ is not topologically minimal.
By using this result, we will prove the following theorem.

\begin{theorem}[Theorem \ref{theorem-mod-invariant-4-main}]\label{theorem-mod-invariant-4}
Let $(\V,\W;F)$ be a weakly reducible, unstabilized Heegaard splitting of genus three in an orientable, irreducible $3$-manifold $M$.
If $F$ is topologically minimal, then the orbit $Mod(M,F).[D]$ of any element of $[D]\in\D(F)$ consists of infinitely many elements.
Moreover, if $F$ is not topologically minimal, then there exist exactly two elements of $\D(F)$ having finite orbits, where either (i) two singleton sets consisting of each element are the only finite orbits or (ii) the set consisting of the two elements is the only finite orbit.
\end{theorem}

\section{Preliminaries\label{section2}}

This section introduces basic notations and the key idea to prove the main theorem.

\begin{definition}
A \textit{compression body} is a $3$-manifold which can be obtained by starting with some closed, orientable, connected surface $F$, forming the product $F\times I$, attaching some number of $2$-handles to $F\times\{1\}$ and capping off all  resulting $2$-sphere boundary components that are not contained in $F\times\{0\}$ with $3$-balls. 
The boundary component $F\times\{0\}$ is referred to as $\partial_+$. 
The rest of the boundary is referred to as $\partial_-$. 
\end{definition}

\begin{definition}
A \textit{Heegaard splitting} of a $3$-manifold $M$ is an expression of $M$ as a union $\V\cup_F \W$, denoted   as $(\V,\W;F)$,  where $\V$ and $\W$ are compression bodies that intersect in a transversally oriented surface $F=\partial_+\V=\partial_+\W$. 
We say $F$ is the \textit{Heegaard surface} of this splitting. 
If $\V$ or $\W$ is homeomorphic to a product, then we say the splitting  is \textit{trivial}. 
If there are compressing disks $V\subset \V$ and $W\subset \W$ such that $V\cap W=\emptyset$, then we say the splitting is \textit{weakly reducible} and call the pair $(V,W)$ a \textit{weak reducing pair}. 
If $(V,W)$ is a weak reducing pair and $\partial V$ is isotopic to $\partial W$ in $F$, then we call $(V,W)$ a \textit{reducing pair}.
If the splitting is not trivial and we cannot take a weak reducing pair, then we call the splitting \textit{strongly irreducible}. 
If there is a pair of compressing disks $(\bar{V},\bar{W})$ such that $\bar{V}$ intersects $\bar{W}$ transversely in a point in $F$, then we call this pair a \textit{canceling pair} and say the splitting is \textit{stabilized}. 
Otherwise, we say the splitting is \textit{unstabilized}.
\end{definition}

\begin{definition}
Let $F$ be a surface of genus at least two in a $3$-manifold $M$. 
Then the \emph{disk complex} $\D(F)$ is defined as follows: 
\begin{enumerate}[(i)]
\item Vertices of $\D(F)$ are isotopy classes of compressing disks for $F$.
\item A set of $m+1$ vertices forms an $m$-simplex if there are representatives for each
that are pairwise disjoint.
\end{enumerate}
We denote the isotopy class of a compressing disk $D$ of $F$ as $[D]$.
\end{definition}

\begin{definition}
Consider a Heegaard splitting $(\V,\W;F)$ of a $3$-manifold $M$. 
Let $\DV(F)$ and $\DW(F)$ be the subspaces of $\D(F)$ spanned by compressing disks in $\V$ and $\W$ respectively. 
We call these subspaces \textit{the disk complexes of $\V$ and $\W$}.
Let $\DVW(F)$ be the subset  of $\D(F)$ consisting of simplices having at least one vertex from $\DV(F)$ and at least one vertex from $\DW(F)$.
If there is no confusion, we will use the disk $V\subset \V$ instead of the isotopy class $[V]\in\DV(F)$ for the sake of convenience.
\end{definition}

Note that (i) each of $\DV(F)$ and $\DW(F)$ is contractible (see \cite{8}) and (ii) $\D(F)=\DV(F)\cup \DVW(F)\cup \DW(F)$, where $\DV(F)\cap\DW(F)=\emptyset$.

\begin{definition}[Bachman, \cite{Bachman2010}]\label{definition-top-minimal}
The \emph{homotopy index} of a complex $\Gamma$ is defined to be 0 if $\Gamma=\emptyset$, and the smallest $n$ such that $\pi_{n-1}(\Gamma)$ is non-trivial, otherwise.
We say a separating surface $F$ with no torus components is \emph{topologically minimal} if its disk complex $\D(F)$ is either empty or non-contractible.
When $F$ is topologically minimal, we define the \emph{topological index} of $F$ as the homotopy index of $\D(F)$.
If $F$ is topologically minimal and weakly reducible, then the topological index is at least two because $\D(F)$ is connected.
Note that if $F$ is topologically minimal and its topological index is two, then $\DVW(F)$ is disconnected (see  the proof of Theorem 2.5 of \cite{Bachman2010}).
\end{definition}

From now on, we will consider only unstabilized Heegaard splittings of an irreducible $3$-manifold. 
If a Heegaard splitting of a compact $3$-manifold has a reducing pair, then the manifold is reducible or the splitting is stabilized (see \cite{SaitoScharlemannSchultens2005}).
Hence, we will exclude the possibilities of reducing pairs among weak reducing pairs.

\begin{definition}
Suppose $W$ is a compressing disk for $F\subset M$. 
Then there is a subset of $M$ that can be identified with $W\times I$ so that $W=W\times\{\frac{1}2\}$ and $F\cap(W\times I)=(\partial W)\times I$. 
We form the surface $F_W$, obtained by \textit{compressing $F$ along $W$}, by removing $(\partial W)\times I$ from $F$ and replacing it with $W\times(\partial I)$. 
We say the two disks $W\times(\partial I)$ in $F_W$ are the $\textit{scars}$ of $W$. 
\end{definition}

\begin{lemma}[Lustig and Moriah, Lemma 1.1 of \cite{7}] \label{lemma-2-8}
Suppose $M$ is an irreducible $3$-manifold and $(\V,\W;F)$ is an unstabilized Heegaard splitting of $M$. 
If $F'$ is obtained by compressing $F$ along a collection of pairwise disjoint disks, then no $S^2$ component of $F'$ can have scars from disks in both $\V$ and $\W$. 
\end{lemma}

\begin{lemma}\label{lemma-2-9}
Suppose $M$ is an irreducible $3$-manifold and $(\V,\W;F)$ is an unstabilized Heegaard splitting of $M$ of genus $g\geq 3$. 
For any component $\mathcal{C}$ of $\DVW(F)$, there exists a weak reducing pair $(V,W)\subset \mathcal{C}$ such that (i) $V$ and $W$ are separating in $\V$ and $\W$ respectively or (ii) $V$ and $W$ are non-separating in $\V$ and $\W$ respectively and $\partial V\cup\partial W$ is separating in $F$.
\end{lemma}

\begin{proof}
If there exists a weak reducing pair in $\mathcal{C}$ such that both disks are non-separating in the relevant compression bodies and the union of their boundaries is separating in $F$, then we get the wanted weak reducing pair.
Hence, assume there is no such weak reducing pair in $\mathcal{C}$.

Let $(V,W)$ be a weak reducing pair in $\mathcal{C}$.
If $V$ and $W$ are separating in $\V$ and $\W$ respectively, then we get the wanted weak reducing pair.
Hence, suppose at least one of them, say $V$, is non-separating in $\V$.
If we compress $F$ along $V$, then we get the genus $g-1$ surface $F_V$.
Let $\bar{V}_1$ and $\bar{V}_2$ be the scars of $V$ in $F_V$.

\Case{a} $W$ is separating in $\W$.

Let us compress $F_V$ along $W$ again and we get the resulting surface $F_{VW}$.
Since $W$ is separating in $\W$, $F_{VW}$ consists of two surfaces $F_{VW}'$ and $F_{VW}''$ and both $\bar{V}_1$ and $\bar{V}_2$ belong to only one of $F_{VW}'$ and $F_{VW}''$, say $F_{VW}'$.
Let $\bar{W}$ be the scar of $W$ in $F_{VW}'$.
Let $\alpha$ be a simple arc embedded in $F_{VW}'$ connecting $\bar{V}_1$ and $\bar{V}_2$ such that $\mathrm{int}(\alpha)$ misses $\bar{V}_1\cup\bar{V}_2\cup \bar{W}$.
Then we can see $\alpha$ realizes a band sum of two parallel copies of $V$ in $\V$, say $V'$, such that $V'$ misses $W$.
That is, $\{V,V',W\}$ forms a $2$-simplex in $\DVW(F)$ containing the weak reducing pair $(V,W)$ and therefore it belongs to $\mathcal{C}$.
Since $V'$ is separating in $\V$, we get the wanted weak reducing pair $(V',W)$ in $\mathcal{C}$.
Note that Lemma \ref{lemma-2-8} guarantees the genus of $F_{VW}'$ is at least one, i.e. we do not need to worry about the possibility that $(V',W)$ would be a reducing pair which cannot exist when the splitting is unstabilized.

\Case{b} $W$ is non-separating in $\W$.

Since $\partial V\cup \partial W$ is non-separating in $F$ by the assumption in the start of the proof, if we compress $F_V$ along $W$, then we get the connected surface $F_{VW}$ of genus $g-2>0$.
Hence, we can find a band-sum of two parallel copies of $V$ in $\V$, say $V'$, and that of $W$ in $\W$, say $W'$, satisfying $V'\cap W'=\emptyset$ up to isotopy, where these band-sums are realized by simple arcs $\alpha_V$ and $\alpha_W$ embedded in $F_{VW}$ respectively such that (i) $\alpha_V$ connects the scars of $V$ and $\alpha_W$ connects the scars of $W$, (ii) $\mathrm{int}(\alpha_V)\cup\mathrm{int}(\alpha_W)$ misses the scars of $V$ and $W$, and (iii) $\alpha_V\cap\alpha_W=\emptyset$.
Therefore, $\{V,V',W', W\}$ forms a $3$-simplex in $\DVW(F)$ containing the weak reducing pair $(V,W)$, i.e. the weak reducing pair $(V',W')$ belongs to $\mathcal{C}$.
Since $V'$ and $W'$ are separating in $\V$ and $\W$ respectively, we get the wanted weak reducing pair $(V',W')$ in $\mathcal{C}$.

This completes the proof.
\end{proof}

Here, we introduce the key idea of this article - the \textit{torus twist}.

\begin{definition}[Johnson, Section 3 of \cite{Johnson2011} or Johnson and Rubinstein, Section 6 of \cite{JohnsonRubinstein2013}]\label{definition-torus-twist}
Suppose $(\V,\W;F)$ is a Heegaard splitting of genus at least three in a $3$-manifold $M$.
Let $V$ and $W$ be compressing disks of $\V$ and $\W$ respectively such that $\partial V$ intersects $\partial W$ transversely in exactly two points.
Then we can say  $\mathcal{T}=N(V\cup W)$, a small neighborhood of $V\cup W$, is a solid torus.
If we give directions to $\partial V$ and $\partial W$, then we can see $\partial V$ and $\partial W$ form an orientation of $F$ at each intersection point.
Suppose the induced orientations at the two points are opposite.
Then $F\cap \mathcal{T}$ is a four punctured sphere (see Figure 2 of \cite{Johnson2011}) and $F\cap\partial \mathcal{T}$ consists of four parallel longitudes in $\partial \mathcal{T}$.
(For the sake of convenience, we can imagine $V$ as a small disk so that $V$ would belong to a product neighborhood of a meridian disk of $\mathcal{T}$ and $W$ as a ``long'' disk as in the upper one of Figure \ref{fig-def-torus-twist}.)
\begin{figure}
\includegraphics[width=13cm]{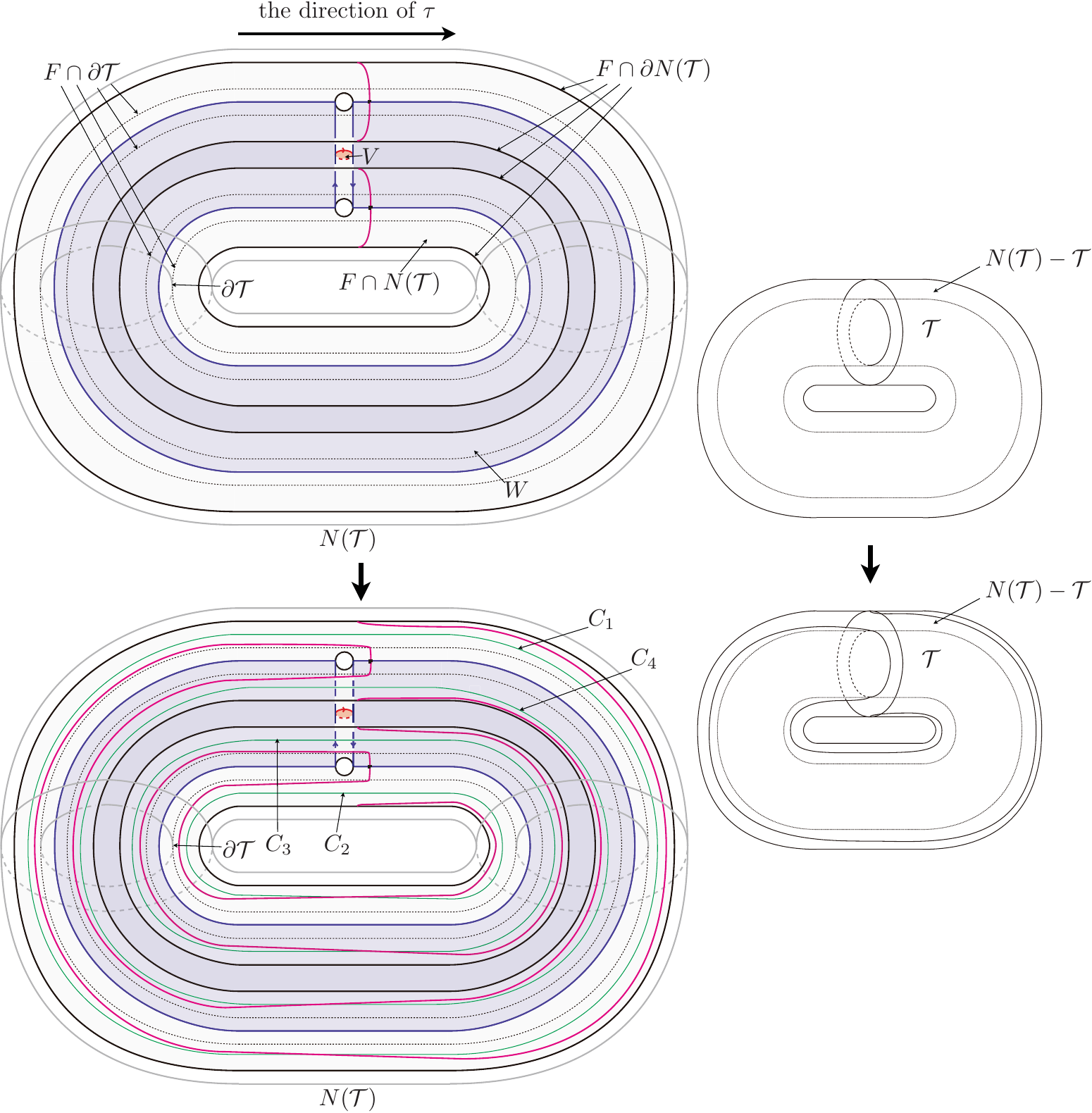}
\caption{A torus twist induces four Dehn-twists on $F$. \label{fig-def-torus-twist}}
\end{figure}
Let us consider a spinning $\tau$ of $\mathcal{T}$ one time along the longitudinal direction such that any point in a loop of $F\cap \partial\mathcal{T}$ goes around the loop once, i.e. each loop of $F\cap \partial\mathcal{T}$ is preserved setwisely during the spinning.
By using the image after the spinning, we can consider $\tau$ as an orientation preserving automorphism of $M$. 
Moreover, we can assume the follows.
\begin{enumerate}
\item $\tau|_{\mathcal{T}}=\mathrm{id}_{\mathcal{T}}$.
\item $\tau|_{\mathrm{cl}(M-N(\mathcal{T}))}=\mathrm{id}_{\mathrm{cl}(M-N(\mathcal{T}))}$ as in the lower one of Figure \ref{fig-def-torus-twist}, where $N(\mathcal{T})$ is a small neighborhood of $\mathcal{T}$ and we take $N(\mathcal{T})$ sufficiently small so that $F\cap N(\mathcal{T})$ would be homeomorphic to $F\cap \mathcal{T}$.
\item $\tau(F)=F$, i.e. $\tau$ is an orientation preserving automorphism of $M$ that takes $F$ onto itself, i.e. $\tau$ is an element of $Mod(M,F)$.
\end{enumerate}
By the assumption, $\tau$ gives a change in $N(\mathcal{T})-\mathcal{T}$ (see the right part of Figure \ref{fig-def-torus-twist}) but it is the identity map elsewhere.

Note that the induced map $\bar{\tau}$ in $Mod(F)$ from $\tau$ consists of Dehn-twists about four curves in $\mathrm{int}(F\cap (N(\mathcal{T})-\mathcal{T}))$ parallel to $F\cap \partial \mathcal{T}$.
We can imagine the four punctured sphere $F\cap N(\mathcal{T})$ as the one obtained by connecting two annuli by a cylinder and then removing the interior of the attaching disks, where the center circle of the cylinder is $\partial V$ or $\partial W$.
Let us name the four curves realizing $\bar{\tau}$  $C_1$, $\cdots$, $C_4$, where $C_1\cup C_4$ belongs to one component of $(F\cap N(\mathcal{T}))-\partial V$,  $C_2\cup C_3$ belongs to the other component, $C_1\cup C_2$ belongs to one component of $(F\cap N(\mathcal{T}))-\partial W$, and $C_3\cup C_4$ belongs to the other component (see the four green curves $C_1$, $\cdots$, $C_4$ in the lower one of Figure \ref{fig-def-torus-twist}).
We consider the Dehn twist about $C_i$ as the Dehn twist in a small neighborhood of $C_i$ in $F$, say $N(C_i)$, and give an $I$-fibration to each $N(C_i)$ so that each $I$-fiber intersects $C_i$ transversely in a point and no $I$-fiber twists around $C_i$.
If we observe the image of an $I$-fiber of $\bar{\tau}$ in $N(C_i)$ for $1\leq i \leq 4$ as in the lower one of Figure \ref{fig-def-torus-twist}, then the image turns left in $N(C_1)$ and $N(C_3)$ and turns right in $N(C_2)$ and $N(C_4)$ from the viewpoint of the one inside $\V$ with respect to the given $I$-fibration
(we can refer to Definition 3.5 of \cite{Yoshizawa2014} for the definition of ``an arc \textit{turns left} or \textit{turns right} in an $I$-fibered annulus'').
The reason is that if $C_i$ and $C_j$ belong to the same component with respect to $\partial V$ or $\partial W$ for $i\neq j$, then the image turns left in one of $N(C_i)$ and $N(C_j)$, say $N(C_i)$, but turns right in $N(C_j)$.
If either (i) the direction of $\tau$ was opposite to that of Figure \ref{fig-def-torus-twist} or (ii) we observe from the viewpoint of the one outside $\V$ instead of that of the one inside $\V$, then we should consider the opposite directions of Dehn twists.
From now on, we will consider the directions of Dehn twists in the sense that we observe from the viewpoint of the one outside $\V$.
We call $\tau$ a \textit{torus twist} defined by the solid torus $\mathcal{T}=N(V\cup W)$.
\end{definition}

\begin{definition}
$Mod(M,F)$ acts on $\D(F)$ naturally as $[g].[D]=[g(D)]\in \D(F)$ for $[g]\in Mod(M,F)$ and $[D]\in \D(F)$.
We define the \textit{orbit} of an element $[D]\in\D(F)$ as
$Mod(M,F). [D]=\{[g].[D]|[g]\in Mod(M,F)\}.$
\end{definition}

\section{The proof of Theorem \ref{theorem-mod-invariant-3}\label{section3}}

In this section, we will prove Theorem \ref{theorem-mod-invariant-3}.
 
\begin{lemma}\label{lemma-disconnected-infinite-I}
Let $M$ be an orientable, irreducible $3$-manifold and $(\V,\W;F)$ an unstabilized Heegaard splitting of $M$ of genus at least three. 
Suppose $(V,W)$ is a weak reducing pair of $(\V,\W;F)$ such that $V$ and $W$ are separating in $\V$ and $\W$ respectively.
If $D$ is a compressing disk of $F$ such that $\partial D$ intersects at least one of $\partial V$ and $\partial W$ up to isotopy in $F$, then the orbit $Mod(M,F).[D]$ of the element $[D]\in \D(F)$ consists of infinitely many elements.
\end{lemma}

\begin{proof}
Without loss of generality, assume $D\subset \W$ and $\partial D$ intersects $\partial V$ up to isotopy.
Note that the compression body containing $D$ itself does not matter in the proof. 

By the assumption that $V$ and $W$ are separating and Lemma \ref{lemma-2-8}, $\partial V$ cuts off a once-punctured surface $F_1$ of genus at least one from $F$ and $\partial W$ separates $\mathrm{cl}(F-F_1)$ into a twice-punctured surface of genus at least one and a once-punctured surface of genus at least one.
Hence, we can find a simple path $\gamma$ connecting $\partial V$ and $\partial W$ in the twice-punctured surface such that $\mathrm{int}(\gamma)$ misses $\partial V\cup\partial W$.
If we drag $\partial W$ along $\gamma$ and push it more into $F_1$, then we get a disk $W'$ which intersects $V$ transversely in exactly two points, i.e. a subarc of $\partial W'$ bounds a bigon $B$ with a subarc of $\partial V$ in $F_1$.
We can see a subarc of $\partial W'$ divides $\mathrm{cl}(F-F_1)$ into two once-punctured surfaces $F_2$ and $F_3$ of genus at least one, where $F_2\cap B$ consists of two points and $F_3\cap B$ consists of a subarc of $\partial V$ (see Figure \ref{fig-torus-twist}).
\begin{figure}
\includegraphics[width=9cm]{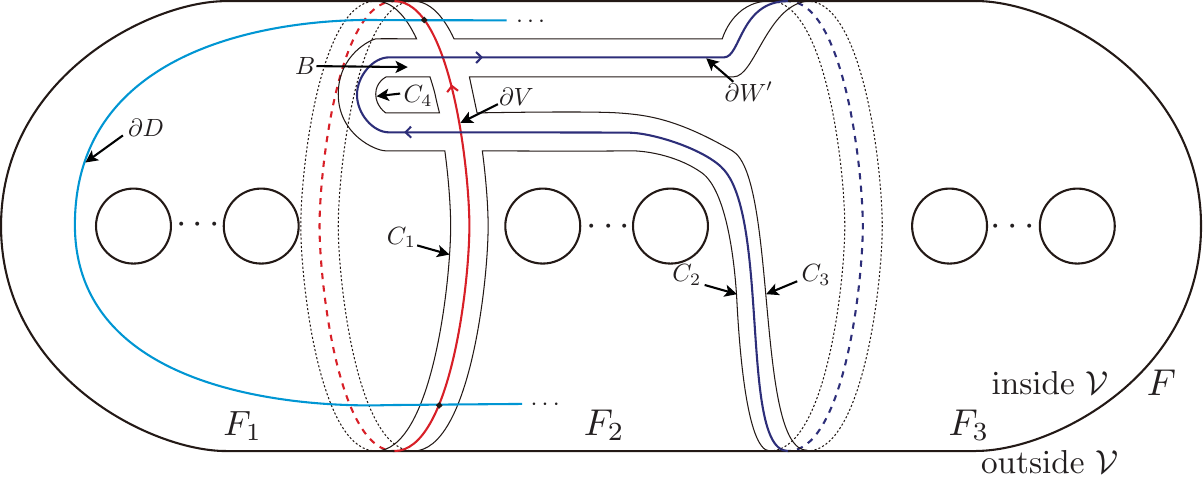}
\caption{the four curves $C_1$, $C_2$, $C_3$ and $C_4$ \label{fig-torus-twist}}
\end{figure}
 
Let $\mathcal{T}$ be the solid torus $N(V\cup W')$ and $\tau$ a torus twist defined by $\mathcal{T}$.
Then $F\cap \mathcal{T}$ is a four-punctured sphere $S$ and the induced map $\bar{\tau}$ in $Mod(F)$ consists of  Dehn twists about four curves in $\mathrm{int}(N(S)-S)$ parallel to $\partial S$, say $C_1$, $\cdots$, $C_4$ (see Figure \ref{fig-torus-twist}), where $N(S)$ is a small neighborhood of $S$ in $F$.
For the sake of convenience, we will use the naming of the four curves and the direction of $\tau$ as in Figure \ref{fig-def-torus-twist} by substituting $V$ and $W'$ for $V$ and $W$ respectively.
Hence, we can assume $\bar{\tau}$ consists of three non-trivial Dehn twists in the disjoint annuli $N(C_1)$, $N(C_2)$ and $N(C_3)$ after discarding the trivial Dehn twist about $C_4$ which bounds a disk in $F$, where each $N(C_i)$ is contained in $\mathrm{int}(F_i)$ for $1\leq i \leq 3$, and therefore $\bar{\tau}$ is the identity map outside $\cup_{j=1}^3 N(C_j)$. 
(Strictly speaking, $\bar{\tau}$ is isotopic to the identity map in $B$ and it is the identity map in $F-(\cup_{j=1}^3 N(C_j)\cup B)$.
But the assumption does not affect our proof.)
Here, we give an $I$-fibration to each $N(C_i)$ as we did in Definition \ref{definition-torus-twist} for $1\leq i \leq 3$.
This means the image of an $I$-fiber of $\bar{\tau}$ turns right in $N(C_1)$ and $N(C_3)$ and turns left in $N(C_2)$ from the viewpoint of the one outside $\V$.

Now we will isotope $D$ so that we could take the simplest image of $\partial D$ of $\bar{\tau}$.

First, we isotope $D$ so that $\partial D$ intersects $(\cup_{j=1}^3 \partial N(C_j))\cup(\partial V\cup\partial W')$ transversely and misses $\partial V\cap\partial W'$.

Let $\bar{F}$ be the closure of the pair of pants component of $F-(\cup_{j=1}^3 N(C_j))$ and $\tilde{F}=\bar{F}\cup(\cup_{j=1}^3 N(C_j))$.
Let $\tilde{C}_i$ ($\bar{C}_i$ resp.) be the component of $\partial N(C_i)$ intersecting $\partial \tilde{F}$ ($\partial \bar{F}$ resp.) for $1\leq i \leq 3$.

We will isotope $D$ so that $\partial D\cap(\cup_{j=1}^3 \partial N(C_j))$ would be minimal as follows.
Let $\tilde{F}_i$ be the closure of the once-punctured component of $F-(\cup_{j=1}^3 \tilde{C}_j)$ adjacent to $\tilde{C}_i$ for $1\leq i \leq 3$.
If there is a subarc of $\partial D$ such that it belongs to $\tilde{F}_i$ ($N(C_i)$ resp.) and it bounds a disk $\Delta$ in $\tilde{F}_i$ ($N(C_i)$ resp.) with a subarc of $\tilde{C}_i$, then we can push this subarc into $\mathrm{int}(N(C_i))$ ($\mathrm{int}(\tilde{F}_i)$ resp.) by using a standard outermost disk argument for $\Delta\cap\partial D$ in $\Delta$ for $1\leq i \leq 3$ (see Figure \ref{fig-outermost}).
\begin{figure}
\includegraphics[width=10cm]{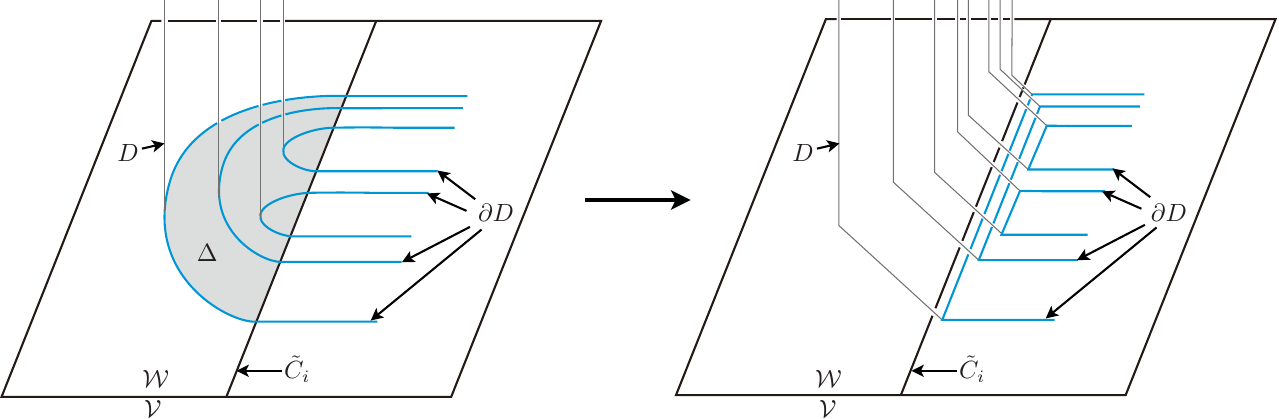}
\caption{the isotopy of $D$ by a standard outermost disk argument \label{fig-outermost}}
\end{figure}
Similarly, if there is a subarc of $\partial D$ such that it belongs to $\bar{F}$ ($N(C_i)$ resp.) and it bounds a disk $\Delta$ in $\bar{F}$ ($N(C_i)$ resp.) with a subarc of $\bar{C}_i$, then we can push this subarc into $\mathrm{int}(N(C_i))$ ($\mathrm{int}(\bar{F})$ resp.) for $1\leq i \leq 3$.
We can see $|\partial D\cap(\cup_{j=1}^3 \partial N(C_j))|$ decreases by two when each outermost disk disappears.
Therefore, we can repeat this argument until there is no such subarc of $\partial D$.
This means every component of $\partial D\cap \bar{F}$, $\partial D\cap \tilde{F}_i$ and $\partial D\cap N(C_i)$ is a properly embedded essential arc in $\bar{F}$, $\tilde{F}_i$ and $N(C_i)$ respectively for $1\leq i \leq 3$ after the isotopies.
Moreover, every component of $\partial D\cap\tilde{F}$ is also a properly embedded essential arc in $\tilde{F}$ because any essential arc in $\bar{F}$ forms an essential arc in $\tilde{F}$ together with two adjacent essential arcs in $\cup_{j=1}^3 N(C_j)$.

In summary, we get Assumption \ref{assumption-1}.

\begin{assumption}\label{assumption-1}
No subarc of $\partial D$ bounds a disk in  $\mathrm{cl}(F-(\cup_{j=1}^3 N(C_j)))$ or $\cup_{j=1}^3 N(C_j)$ with a subarc of $\cup_{j=1}^3 \partial N(C_j)$.
\end{assumption}

Since $\bar{F}$ is a pair of pants, there exist at most three isotopy classes of mutually disjoint essential arcs in $\bar{F}$.
Moreover, there are at most two isotopy classes among them such that they connect (i) $\bar{C}_1$ and $\bar{C}_2$ or (ii) $\bar{C}_3$ and $\bar{C}_2$.
Hence, we isotope $D$ as follows.
\begin{enumerate}
\item This isotopy affects only $\mathrm{int}(\bar{F})$ in $F$,\label{pp0}
\item $\partial D$ misses $B$,\label{pp2} 
\item if a component of $\partial D\cap \bar{F}$ connects $\bar{C}_1$ and $\bar{C}_2$, then it intersects $\partial V$ in exactly one point and it misses $\partial W'$, and
\item if a component of $\partial D\cap \bar{F}$ connects $\bar{C}_3$ and $\bar{C}_2$, then it intersects $\partial W'$ in exactly one point and it misses $\partial V$.
\end{enumerate}
An easy way to imagine this isotopy is to deform $\partial V\cup\partial W'$ instead of moving $\partial D$.
That is, $\partial V\cup\partial W'$ shrinks into a small neighborhood of $\bar{C}_1\cup\gamma\cup \bar{C}_3$ in $\bar{F}$ as in Figure \ref{fig-pants-A}, where $\gamma$ is a path from $\bar{C}_1$ to $\bar{C}_3$ missing the previous two isotopy classes of $\partial D\cap \bar{F}$. 
\begin{figure}
\includegraphics[width=12cm]{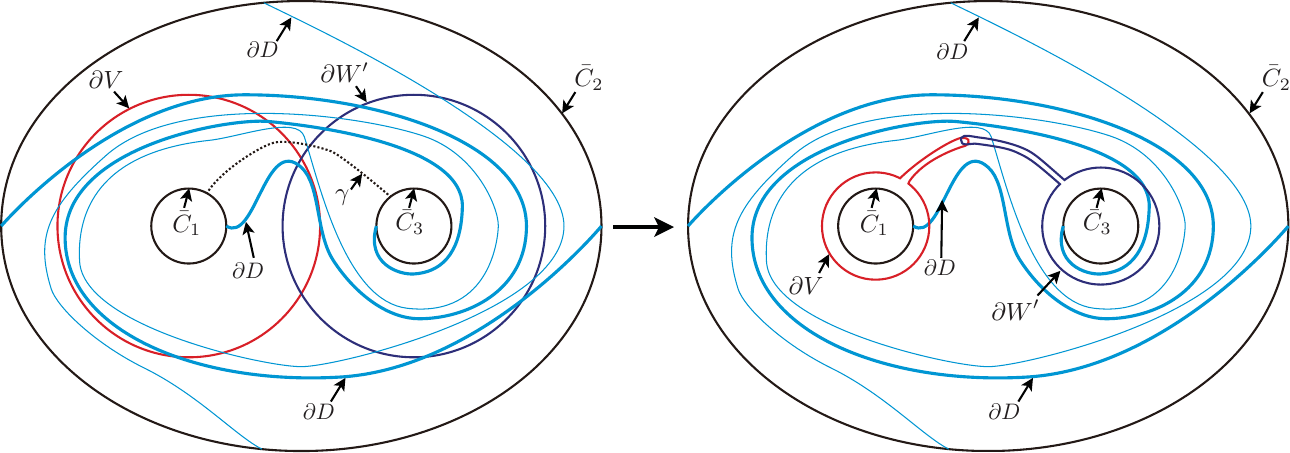}
\caption{An easy way to imagine the isotopy of $\partial D$ \label{fig-pants-A}}
\end{figure}

Next, if $\partial D \cap \tilde{F}$ has a twist about $\tilde{C}_i$ in $ \tilde{F}$ and this starts from $\tilde{C}_i$ for $1\leq i \leq 3$, then we isotope $D$ so that the isotopy pushes such twists out from $\tilde{F}$ as in Figure \ref{fig-push-twists}.
Therefore, we get Assumption \ref{assumption-2}.

\begin{assumption}\label{assumption-2}
No component of $\partial D\cap \tilde{F}$ has a twist about $\tilde{C}_i$ starting from $\tilde{C}_i$ for $1\leq i \leq 3$.
\end{assumption}

If we make sure $|\partial D\cap\tilde{C}_i|$ and $|\partial D\cap\bar{C}_i|$ are unchanged during the isotopy for $1\leq i \leq 3$, then  the shapes of components of $F-(\partial D\cup  (\cup_{j=1}^3 \partial N(C_j))$ are also unchanged, i.e. the minimality of $\partial D\cap(\cup_{j=1}^3 \partial N(C_j))$ is unchanged during the isotopy.

\begin{figure}
\includegraphics[width=9cm]{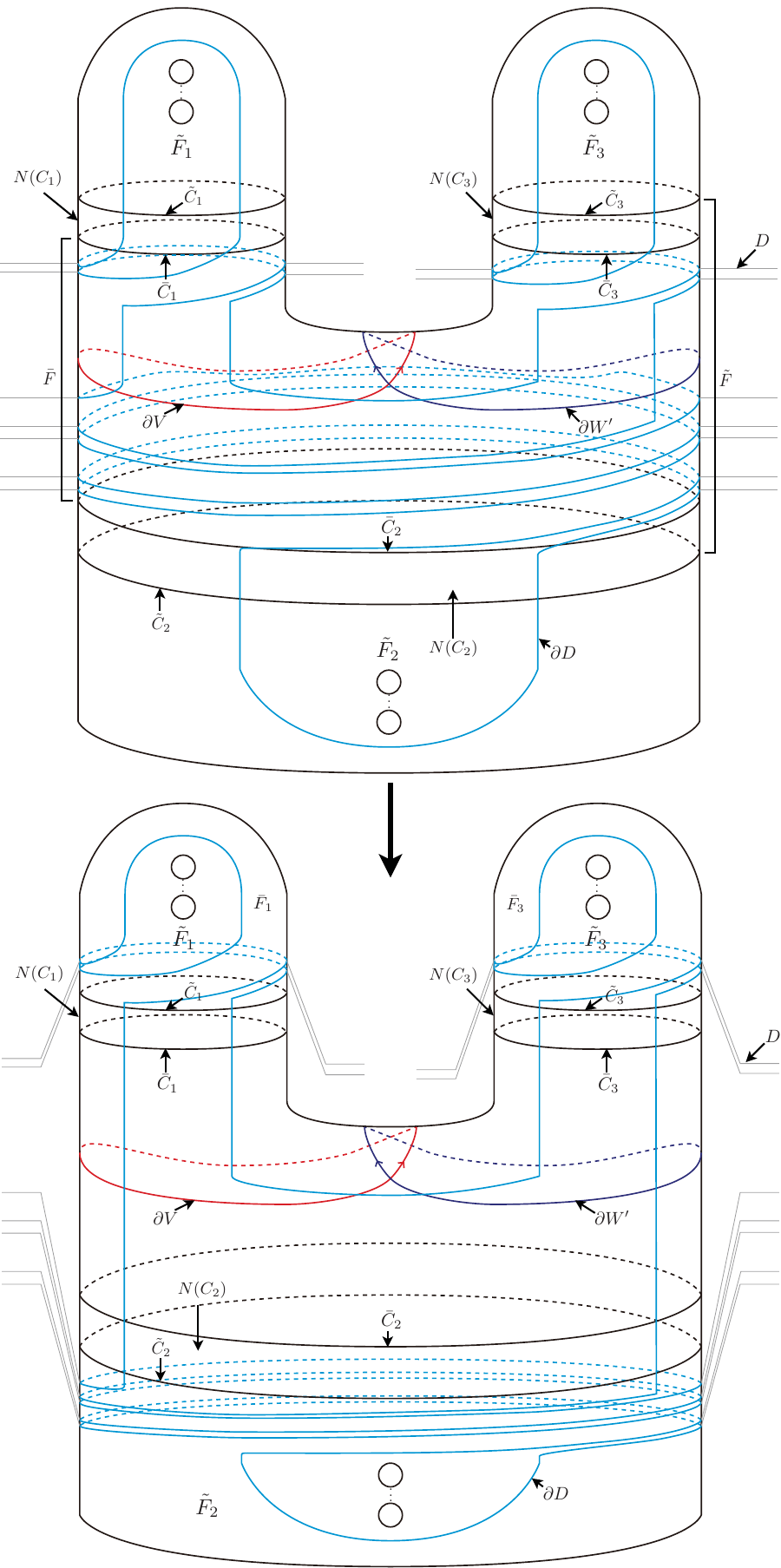}
\caption{push the twists about $\tilde{C}_i$ starting from $\tilde{C}_i$ outside $\tilde{F}$ \label{fig-push-twists}}
\end{figure}

By Assumption \ref{assumption-2} and  Assumption \ref{assumption-1}, we can assume each component of $\partial D\cap N(C_i)$ is an $I$-fiber of the given $I$-fibration of $N(C_i)$ for $1\leq i\leq 3$ by using a small isotopy of $D$.

By the assumption that $\partial D$ intersects $\partial V$ up to isotopy, $\partial D\cap F_1$ must have at least one properly embedded essential arc in $F_1$.
Since the region between $\partial F_1(=\partial V)$ and $\partial \tilde{F}_1(=\tilde{C}_1)$ is just an annulus, such arc in $F_1$ must intersect $\tilde{F}_1$, i.e. $\partial D\cap N(C_1)\neq \emptyset$ up to isotopy.

Let us do the torus twist $\tau$ twice (we will see the reason why we use $\tau^2$ instead of $\tau$ in the proof of Claim \ref{claim-2}).
For the sake of convenience, let $\nu$ be $\tau^2$. 
Then we can see $\nu(D)$ is a compressing disk of $\W$ because $\tau(\W)=\W$.
Let $\bar{\nu}$ be the induced map in $Mod(F)$ from $\nu$.
In order to compare $\bar{\nu}(\partial D)$ with $\partial D$ as two curves in $F$, we assume we prepared a disjoint parallel copy of $D$ in $\W$ near $D$ in advance and consider the image of this copy of $\nu$ as $\nu(D)$.
We take the copy of $D$ in a sufficiently small neighborhood of $D$ in $\W$ so that $\bar{\nu}(\partial D)$ also would not intersect $B$.
Hence, we obtain $\bar{\nu}(\partial D)$ by starting from a parallel, disjoint copy of $\partial D$ in a small neighborhood of $\partial D$ in $F$ and then replacing the parts intersecting $\cup_{j=1}^3 N(C_j)$ by the Dehn twisted subarcs as in the upper one of Figure \ref{fig-high-twists}.
Therefore, we get Assumption \ref{assumption-3}.

\begin{assumption}\label{assumption-3}
A component of $\bar{\nu}(\partial D)-(\cup_{j=1}^3 N(C_j))$ is parallel to the corresponding component of $\partial D-(\cup_{j=1}^3 N(C_j))$ in $F-(\cup_{j=1}^3 N(C_j))$ through a uniquely determined rectangle component of $(F-(\cup_{j=1}^3 N(C_j)))-(\partial D\cup \bar{\nu}(\partial D))$ and vise versa, where two opposite edges of the rectangle come from $\cup_{j=1}^3 \partial N(C_j)$ and the other opposite edges come from $\partial D$ and $\bar{\nu}(\partial D)$ respectively.
\end{assumption}

Hence, $\bar{\nu}(\partial D)$ intersects $\partial D$ transversely in this setting.

Since $\bar{\nu}\in Mod(F)$ and $\bar{\nu}$ is the identity map outside $\cup_{j=1}^3 N(C_j)$, we get Assumption \ref{assumption-4} from Assumption \ref{assumption-1} directly.

\begin{assumption}\label{assumption-4}
No subarc of $\bar{\nu}(\partial D)$ bounds a disk in  $\mathrm{cl}(F-(\cup_{j=1}^3 N(C_j)))$ or $\cup_{j=1}^3 N(C_j)$ with a subarc of $\cup_{j=1}^3 \partial N(C_j)$.
\end{assumption}

Moreover, we get Assumption \ref{assumption-5}.

\begin{assumption}\label{assumption-5}
If $\bar{\nu}(\partial D)$ ($\partial D$ resp.) once comes into $N(C_i)$ for some $i$, then it must intersect every component of $\partial D \cap N(C_i)$ ($\bar{\nu}(\partial D) \cap N(C_i)$ resp.) before it leaves $N(C_i)$.
Indeed, every point of $\partial D\cap \bar{\nu}(\partial D)$ belongs to $\mathrm{int}(\cup_{j=1}^3 N(C_j))$.
\end{assumption}

\begin{figure}
\includegraphics[width=9cm]{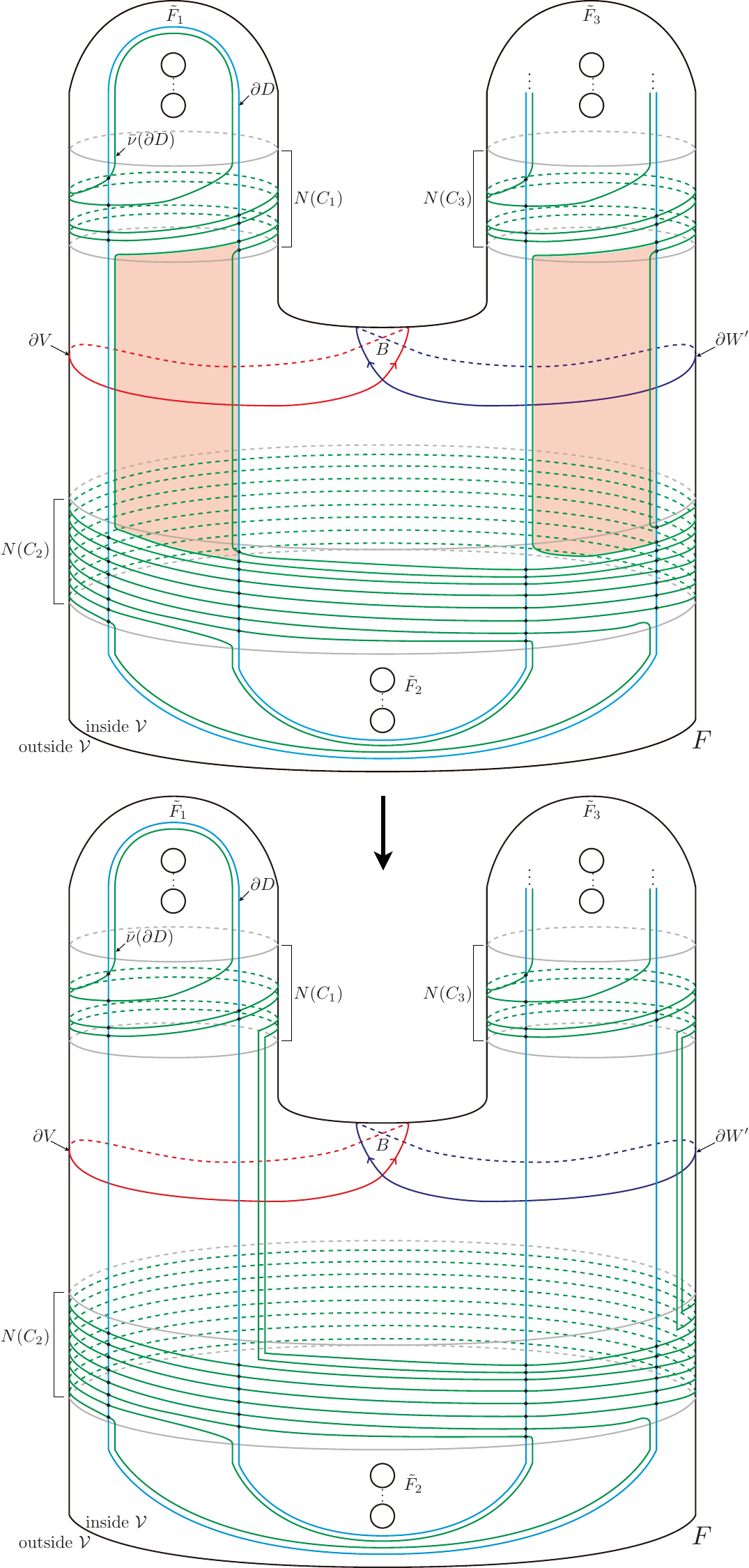}
\caption{removing the bigon regions  \label{fig-high-twists}}
\end{figure}

The next claim is the main part of Lemma \ref{lemma-disconnected-infinite-I}.

\begin{claim}\label{claim-1}
The geometric intersection number $i(\partial D,\bar{\nu}(\partial D))>0$ and it is realized by a sequence of isotopies of $\nu(D)$ in $\W$.
\end{claim}

\begin{proofN}{Claim \ref{claim-1}}
It is sufficient to show that we can isotope $\nu(D)$ in $\mathcal{W}$ so that $\bar{\nu}(\partial D)$ intersects $\partial D$ transversely in at least one point in $F$ and there is no bigon region between them by the \textit{bigon criterion} (see Proposition 1.7 of \cite{PrimerMCGS}).

Suppose there is no bigon region between $\partial D$ and $\bar{\nu}(\partial D)$.
Since $\partial D\cap N(C_1)\neq\emptyset$ up to isotopy, we can guarantee at least one intersection point between $\partial D$ and $\bar{\nu}(\partial D)$ by Assumption \ref{assumption-5}, leading to the result.

Hence, assume there is a bigon region $R$ between $\partial D$ and $\bar{\nu}(\partial D)$ in $F$.
Suppose there is a component of $F-(\partial D \cup \bar{\nu}(\partial D))$ entirely contained in some $N(C_i)$.
Then the component must be a $4$-gon, where two opposite edges come from $\partial D$ and two opposite edges come from $\bar{\nu}(\partial D)$.
Hence, $R$ cannot be contained entirely in any $N(C_j)$ for $1\leq j \leq 3$. 
Therefore, $R$ satisfies one of the follows by Assumption \ref{assumption-5}.
\begin{enumerate}
\item $R\subset\tilde{F}_i\cup N(C_i)$ for some $1\leq i \leq 3$.\label{p1}
\item $R\subset\bar{F}\cup N(C_i)$ for some $1\leq i \leq 3$.\label{p2}
\item $R\subset N(C_1)\cup\bar{F}\cup N(C_3)$, $R\cap N(C_1)\neq\emptyset$, and $R\cap N(C_3)\neq\emptyset$.\label{p3}
\item $R\subset N(C_i)\cup\bar{F}\cup N(C_2)$, $R\cap N(C_i)\neq\emptyset$, and $R\cap N(C_2)\neq\emptyset$ for $i=1$ or $3$.\label{p4}
\end{enumerate}
We will prove the only possible case is (\ref{p4}). 

Let $\partial R=\delta\cup\delta'$, where $\delta\subset \partial D$ and $\delta'\subset \bar{\nu}(\partial D)$.

Suppose $R$ intersects only one $N(C_i)$ for some $i$, i.e. we consider the cases (\ref{p1}) and (\ref{p2}).

Suppose $\delta\subsetneq N(C_i)$, i.e. $|\partial N(C_i)\cap \delta|\geq 2$.
If there are adjacent two points of $\partial N(C_i)\cap \delta$ in $\delta$, then either (i) $\delta$  comes into $N(C_i)$ and then leaves $N(C_i)$  or (ii) $\delta$ leaves  $N(C_i)$ and then comes into $N(C_i)$.
Hence, if $|\partial N(C_i)\cap \delta|\geq 3$, then $\partial D$ comes into $N(C_i)$ and then leaves $N(C_i)$ in the middle of $\delta$ and therefore there must be at least one more intersection point of $\partial D\cap \bar{\nu}(\partial D)$ in the middle of $\delta$ by Assumption \ref{assumption-5}, violating the assumption that $R$ is a bigon.
This means $|\partial N(C_i)\cap \delta|=2$  and symmetrically we get $|\partial N(C_i)\cap \delta'|=0$ or $2$.
Let us start from one intersection point of $\delta\cap\delta'$ and follow $\delta$.
Then $\delta$ passes through a component $\sigma$ of $\partial N(C_i)$, it returns to $\sigma$ again, and we finally ends at the other intersection point of $\delta\cap\delta'$.  
Here, $\delta$ divides $\sigma$ into two subarcs $\sigma_1$ and $\sigma_2$.
If one of $\sigma_1$ and $\sigma_2$ belongs to $R-\delta'$, then it bounds a disk in $R$ with a subarc of $\delta\subset\partial D$, violating Assumption \ref{assumption-1}.
Therefore, if $\sigma$ once comes into $R$ by passing through $\delta$, then it leaves $R$ by passing through $\delta'$, i.e. $\sigma\cap \delta'$ consists of the two points corresponding to the two intersection points $\sigma\cap \delta$.
Hence, $\sigma$ divides $R$ into two triangles in $N(C_i)$ and one rectangle belonging to a component of $\mathrm{cl}(F-(\cup_{j=1}^3 N(C_j)))$, where the boundary of each triangle consists of three edges coming from $\delta$, $\delta'$ and $\sigma$ respectively and the boundary of the rectangle consists of two opposite edges coming from $\sigma$ and the other two opposite edges coming from $\delta$ and $\delta'$ respectively.
Since $\delta'\subset \bar{\nu}(\partial D)$, both edges of these two triangles belonging to $\delta'$ either turn left or turn right in $N(C_i)$.
This means these two edges are positioned in the same direction with respect to the $I$-fibers intersecting $\delta$ in $N(C_i)$ (see (a) of Figure \ref{fig-Mobius}).
\begin{figure}
\includegraphics[width=12cm]{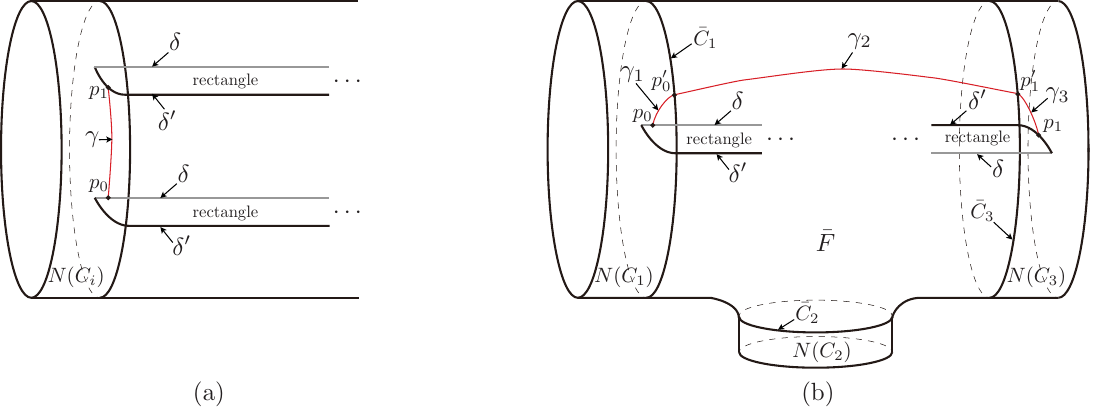}
\caption{the triangles and the rectangle in $R$ \label{fig-Mobius}}
\end{figure}
Hence, there are two points $p_0\in\mathrm{int}(\delta\cap  N(C_i))$ and $p_1\in\mathrm{int}(\delta'\cap N(C_i))$ such that (i) they are connected by a simple path $\gamma\subset \mathrm{int}(N(C_i))$, (ii) $\mathrm{int}(\gamma)$ misses $R$, and (iii) $\mathrm{int}(\gamma)$ misses the two $I$-fibers of $N(C_i)$ intersecting $\delta$.
Let $N(\gamma)$ be a small neighborhood of $\gamma$ in $N(C_i)$.
Then we can see $N(\gamma)\cup R$ is a Mobius band embedded in $F$, violating the assumption that $F$ is an orientable surface.

Hence, $\delta\subset N(C_i)$, i.e. it is a subarc of an $I$-fiber of $N(C_i)$.
Indeed, $\delta$ belongs to $\mathrm{int}(N(C_i))$ by Assumption \ref{assumption-5}.
If $\delta'\subset N(C_i)$, then $\partial R\subset N(C_i)$ and therefore either $R\subset N(C_i)$ or $R$ contains $F-N(C_i)$, leading to a contradiction.
Hence, $\delta'\subsetneq N(C_i)$.
In this case, $\partial N(C_i)$ cuts off an outermost disk $\Delta$ from $R$ such that $\partial\Delta$ consists of a subarc of $\delta'$ and a subarc of $\partial N(C_i)$ because $\delta\cap \partial N(C_i)=\emptyset$.
But this violates Assumption \ref{assumption-4}.
Hence, we have ruled out the cases (\ref{p1}) and (\ref{p2}).

Let us consider the case (\ref{p3}), i.e. $R\subset N(C_1)\cup\bar{F}\cup N(C_3)$, $R\cap N(C_1)\neq\emptyset$, and $R\cap N(C_3)\neq\emptyset$.
In this case, each of $\delta$ and $\delta'$ intersects both $N(C_1)$ and $N(C_3)$.
Since $R$ is a bigon, $|(\partial N(C_1)\cup\partial N(C_3))\cap \delta|=2$ and  $|(\partial N(C_1)\cup\partial N(C_3))\cap \delta'|=2$ by the similar argument using Assumption \ref{assumption-5} in the previous case.
Moreover, if any component of $\partial N(C_1)\cup\partial N(C_3)$ comes into $R$ by passing through $\delta$, then it must leave $R$ by passing through $\delta'$ by Assumption \ref{assumption-1}.
Hence, $\partial N(C_1)\cup\partial N(C_3)$ divides $R$ into two triangles in $N(C_1)$ and $N(C_3)$ respectively and one rectangle in $\bar{F}$.
Choose four points $p_0\in\mathrm{int}(\delta \cap N(C_1))$, $p_0'\in \mathrm{int}(\bar{C}_1-R)$, $p_1'\in \mathrm{int}(\bar{C}_3-R)$ and $p_1\in\mathrm{int}(\delta'\cap  N(C_3))$.
Then we can connect $p_0$ and $p_0'$ by a simple path $\gamma_1\subset N(C_1)$ such that (i) $\gamma_1$ turns left in $N(C_1)$ and (ii) $\mathrm{int}(\gamma_1)$ misses $R$ because $\delta'$ turns right in $N(C_1)$ (see (b) of Figure \ref{fig-Mobius}).
Since the complement of $R$ in $\bar{F}$ is connected by the assumption that $R\cap \bar{F}$ is a rectangle such that two opposite edges belong to $\bar{C}_1$ and $\bar{C}_3$ respectively, we can connect $p_0'$ and $p_1'$ by a simple path $\gamma_2$ in $\bar{F}-R$ such that $\mathrm{int}(\gamma_2)$ misses $\cup_{j=1}^3 \bar{C}_j$.
Moreover, we can connect $p_1'$ and $p_1$ by a simple path $\gamma_3\subset N(C_3)$ such that (i) $\gamma_3$ turns right in $N(C_3)$ and (ii) $\mathrm{int}(\gamma_3)$ misses $R$ because $\delta'$ also turns right in $N(C_3)$.
Let $\gamma$ be $\cup_{j=1}^3\gamma_j$.
Then we can see $N(\gamma)\cup R$ is a Mobius band embedded in $F$, violating the assumption that $F$ is an orientable surface.
This means we also have ruled out the case (\ref{p3}). 

Hence, we conclude only the case (\ref{p4}) holds, i.e. $R\subset N(C_i)\cup\bar{F}\cup N(C_2)$, $R\cap N(C_i)\neq\emptyset$, and $R\cap N(C_2)\neq\emptyset$ for $i=1$ or $3$.
Let us consider the pair of pants $\tilde{F}$ and the families of properly embedded arcs $\mathcal{A}$ and $\mathcal{A}'$ in $\tilde{F}$, where $\mathcal{A}=\{$the components of $\partial D\cap \tilde{F}$ connecting $\tilde{C}_i$ and $\tilde{C}_2$ for $i=1$ or $3\}$ and $\mathcal{A}'=\{$the components of $\bar{\nu}(\partial D)\cap \tilde{F}$  connecting $\tilde{C}_i$ and $\tilde{C}_2$ for $i=1$ or $3\}$.
By the existence of the bigon $R$, $\mathcal{A}\neq \emptyset$ and $\mathcal{A}'\neq \emptyset$.
Note that there might be a component of $\partial D \cap \tilde{F}$ not belonging to $\mathcal{A}$ and therefore there also might be a component of $\bar{\nu}(\partial D)\cap \tilde{F}$ not belonging to $\mathcal{A}'$.
But we can ignore such components of $\partial D \cap \tilde{F}$ and $\bar{\nu}(\partial D)\cap \tilde{F}$ because they cannot intersect the bigon $R$ by the assumption of the case (\ref{p4}).

Let $\Gamma$ be the circle $\mathrm{cl}((\partial V\cup\partial W') - \partial B)$.
Since we isotoped $D$ so that every element of $\mathcal{A}$ intersects $\partial V\cup\partial W'$  transversely in exactly one point before Claim \ref{claim-1} and $\partial D$ satisfies Assumption \ref{assumption-2}, we can take a height function $h:\tilde{F}\to[0,1]$ satisfying the follows (see Figure \ref{fig-height}).
\begin{figure}
\includegraphics[width=6cm]{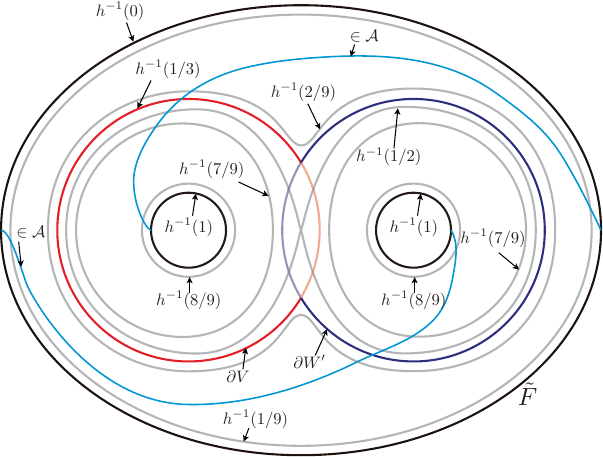}
\caption{a height function of $\tilde{F}$ \label{fig-height}}
\end{figure}
\begin{enumerate}
\item $\mathcal{A}$ consists of at most two families of parallel, vertical line segments.
\item $h^{-1}(0)=\tilde{C}_2$ and $h^{-1}(\frac{1}{3})=\Gamma$,
\item $h^{-1}(t)$ consists of a circle for $t\in[0,1/2)$,
\item $h^{-1}(\frac{1}{2})$ has the ``$\infty$''-shape such that the saddle point is the center point of ``$\infty$''.
\item $h^{-1}(t)$ consists of the union of two disjoint circles for $t\in(\frac{1}{2},1]$, 
\item $h^{-1}(t)\cap (\partial V\cup\partial W')=\emptyset$ for $t\in[\frac{2}{3},1]$, and
\item $h^{-1}(1)=\tilde{C}_1\cup\tilde{C}_3$.
\end{enumerate}

By Assumption \ref{assumption-3}, we can assume $(\mathcal{A}\cup\mathcal{A}')\cap \bar{F}$ consists of at most two families of parallel, vertical line segments in $\bar{F}$.
Moreover, we can also assume every component of $\mathcal{A}'\cap N(C_i)$ is monotone with respect to $h$ for $1\leq i \leq 3$ because $\mathcal{A}\cap N(C_i)$ is vertical and  $\mathcal{A}'\cap N(C_i)$ is the Dehn-twisted image of a parallel copy of $\mathcal{A}\cap N(C_i)$ for $1\leq i \leq 3$. 
Therefore, it is convenient to imagine $\delta'$ in $\tilde{F}$ as a line segment having a long, vertical middle whose ends bend toward the same direction whereas $\delta$ is a vertical line segment by the assumption that $\bar{\nu}(\partial D)$ turns right in $N(C_i)$ for $i=1,3$ and turns left in $N(C_2)$.
This means the bigon $R$ consists of two triangles in $N(C_i)$ and $N(C_2)$ respectively for $i=1$ or $3$ and a rectangle in $\bar{F}$.
Hence, we can remove any bigon region as in Figure \ref{fig-bigon-removing} by an isotopy of $\nu(D)$, where this isotopy can be represented by a \textit{horizontal isotopy} of  $\bar{\nu}(\partial D)$ in $\tilde{F}$.
But there might appear new bigons after we remove the ``\textit{first-step}'' bigons, i.e. the bigons can be found initially.
Hence, we need to consider a sequence of isotopies of $\nu(D)$ inductively as follows.
\begin{enumerate}
\item Remove all bigons appearing in this step by the corresponding horizontal isotopies of $\mathcal{A}'$.
Here, $|\partial D\cap\bar{\nu}(\partial D)|$ decreases by two when each bigon disappears.
\begin{figure}
\includegraphics[width=6cm]{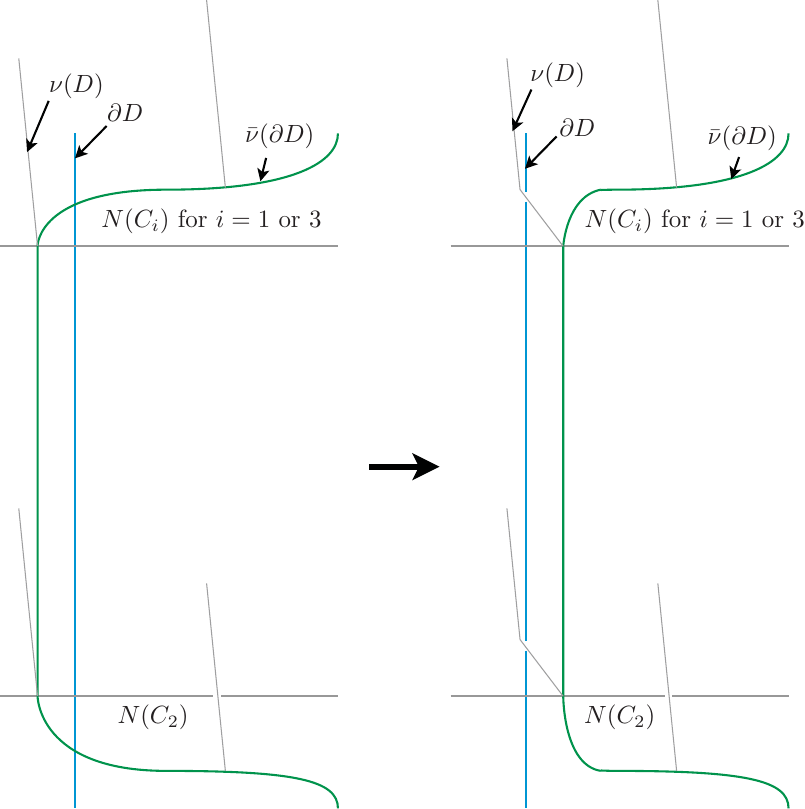}
\caption{the bigon-removing horizontal isotopy \label{fig-bigon-removing}}
\end{figure}
We assume each of these isotopies only affects a small neighborhood of the corresponding bigon so that only one element of $\mathcal{A}$ intersects the neighborhood of the bigon and also only one element of $\mathcal{A}'$ does.
It is convenient to imagine these horizontal isotopies just loosen up the windings of $\bar{\nu}(\partial D)$ in $\tilde{F}$ slightly.
If we make sure $|\bar{\nu}(\partial D)\cap\bar{C}_i|$ is unchanged during each horizontal isotopy for $1\leq i \leq 3$, then the shapes of components of $F-(\bar{\nu}(\partial D)\cup (\cup_{j=1}^3 \partial N(C_j)))$ are also unchanged, i.e. Assumption \ref{assumption-4} is unchanged during the horizontal isotopies.
\item We will find new bigons after these horizontal isotopies.
If there exists a new bigon, then we can consider the \textit{preimage} $\bar{R}$ of the new bigon in the sense that it is the region of $\tilde{F}$ which becomes the new bigon after the horizontal isotopies removing old bigons.
If we ignore the vertices from old bigons, then we can regard $\bar{R}$ as a big bigon.
Hence, we will say  $\partial \bar{R}$ consists of two edges in this sense. 
Then either (i) exactly one edge of $\partial \bar{R}$ has been affected by old bigons or (ii) both edges have been affected by old bigons.
That is, either (i) there is a $(\geq4)$-gon forming $\bar{R}$ with at least one old bigon (see (A) and (B) of Figure \ref{fig-gon}) or (ii) there is a $(\geq6)$-gon forming $\bar{R}$ with at least two old bigons (see (C) of Figure \ref{fig-gon}).
\item For the $(\geq4)$-gon case, at least four vertices of the $(\geq4)$-gon come from a subarc $\alpha$ of $\partial D$ or $\bar{\nu}(\partial D)$ which is an edge of $\bar{R}$ and at least one old bigon $R$  intersects $\alpha$ in the middle (see (A) or (B) of Figure \ref{fig-gon} respectively).
Let $\alpha'$ be the other edge of $\bar{R}$.
\begin{figure}
\includegraphics[width=12cm]{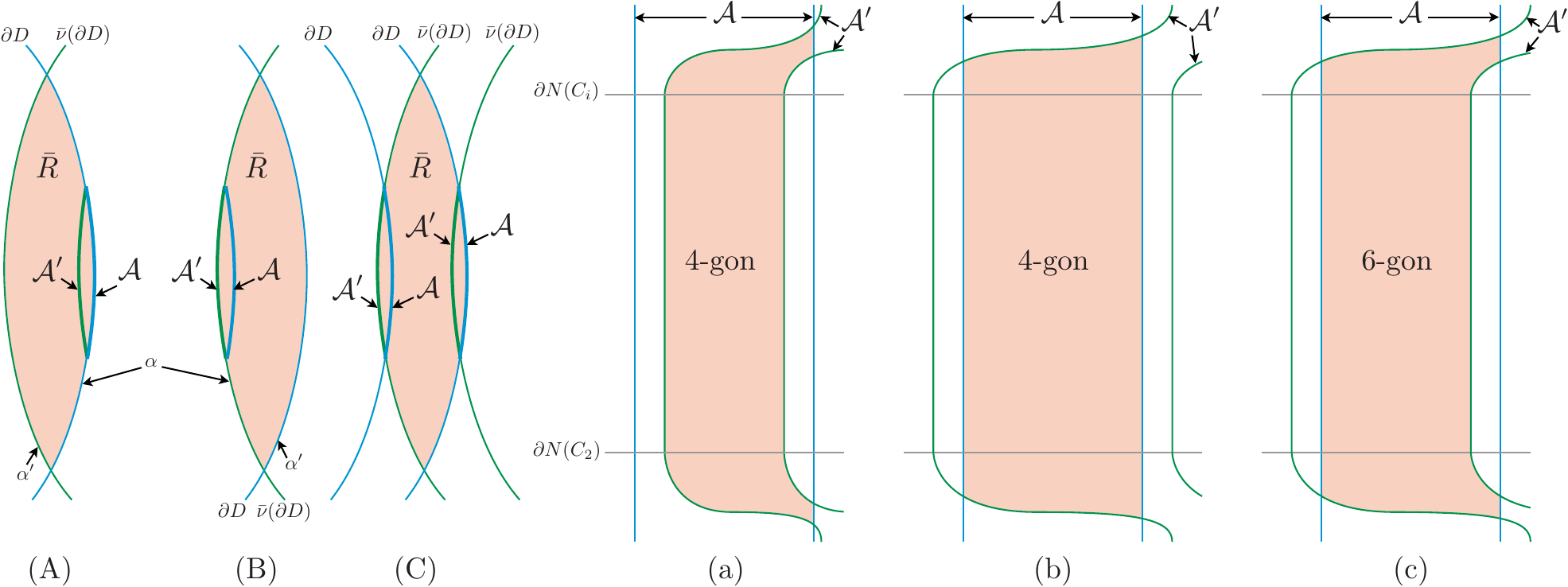}
\caption{possible $4$-gons and $6$-gon  \label{fig-gon}}
\end{figure}
Since $\alpha$  is an extension of an edge of $R$, the extended part in $\alpha$ from the edge of $R$ also belongs to an element of $\mathcal{A}$ or $\mathcal{A}'$ until it leaves $\tilde{F}$.
But we will show $\alpha$  itself belongs to an element of $\mathcal{A}$ or $\mathcal{A}'$ in Claim \ref{claim-2} and therefore we admit Claim \ref{claim-2} for the preimages in this step.

Here, $\cup_{j=1}^3 \partial N(C_j)$ intersects $\alpha$ in at least two points by the existence of $R$.
If we consider Assumption \ref{assumption-1} and Assumption \ref{assumption-4}, then we conclude each component of $\bar{R}\cap(\cup_{j=1}^3 \partial N(C_j))$ must intersect both $\alpha$ and $\alpha'$.
Moreover, there are at least one component of $\bar{R}\cap(\cup_{j=1}^3 \partial N(C_j))$ coming from $\bar{C}_2$ and at least one component coming from $\bar{C}_i$ for $i=1$ or $3$ by the shape of $R$.
This means they cut $\bar{R}$ into two triangles intersecting $N(C_2)$ and $N(C_i)$ respectively and one rectangle between the triangles.
In particular, there is no another component of $\bar{R}\cap(\cup_{j=1}^3 \partial N(C_j))$ in the rectangle by considering the way how $R$ intersects $\bar{C}_2$ and $\bar{C}_i$, i.e. if there is another component, then it must intersects one of the triangles.
This means the rectangle itself belongs to $\bar{F}$. 
Moreover, if $\partial D$ or $\bar{\nu}(\partial D)$ once comes into $N(C_j)$ by passing through one component of $\partial N(C_j)$ for any $1\leq j \leq 3$, then it must leave $N(C_j)$ by passing through the other component by Assumption \ref{assumption-1} or Assumption \ref{assumption-4} respectively, i.e. if  $|\bar{R}\cap(\cup_{j=1}^3 \partial N(C_j))|>2$, then at least one of the triangles intersects both components of $\partial N(C_j)$ for some $1\leq j \leq 3$ and therefore $\alpha\subsetneq \tilde{F}$, violating Claim \ref{claim-2} for this step.
If there is another old bigon intersecting $\alpha$ other than $R$, then $|\bar{R}\cap(\cup_{j=1}^3 \partial N(C_j))|>2$ and therefore we conclude the number of old bigons affecting $\alpha$ is exactly one.
This means the new bigon also consists of two triangles in $N(C_2)$ and $N(C_i)$ respectively for $i=1$ or $3$ and a rectangle in $\bar{F}$ as  old bigons do.
In summary, the new bigon appears between an element of $\mathcal{A}$ and an (possibly horizontally isotoped) element of $\mathcal{A}'$ (see (a) or (b) of Figure \ref{fig-gon}).

\item For the $(\geq6)$-gon case, at least four vertices of the $(\geq6)$-gon come from a subarc $\alpha$ of $\partial D$ or $\bar{\nu}(\partial D)$ which is an edge of $\bar{R}$ and at least two vertices come from the interior of a subarc $\alpha'$ of $\bar{\nu}(\partial D)$ or $\partial D$ which is the other edge of $\bar{R}$ respectively.
Moreover, at least one old bigon intersects $\alpha$ in the middle and at least one old bigon intersects $\alpha'$ in the middle (see (C) of Figure \ref{fig-gon}).
Let us admit Claim \ref{claim-2} for the preimages in this step and suppose $\alpha$ is the edge of Claim \ref{claim-2} without loss of generality.
If we use the same argument in the previous case, then we conclude $|\bar{R}\cap(\cup_{j=1}^3 \partial N(C_j))|=2$ and therefore each of $\alpha$ and $\alpha'$ is affected by only one old bigon.
This means the new bigon consists of two triangles in $N(C_2)$ and $N(C_i)$ respectively for $i=1$ or $3$ and a rectangle in $\bar{F}$ as old bigons do.
In summary, the new bigon appears between an element of $\mathcal{A}$ and a horizontally isotoped element of $\mathcal{A}'$ (see (c) of Figure \ref{fig-gon}).
In particular, these two old bigons intersect the same $N(C_i)$ for $i=1$ or $3$.

\item In summary, any new bigon in the next step also can be removed by a horizontal isotopy if Claim \ref{claim-2} is true for the preimages in this step.
\end{enumerate}

Since every element of $\mathcal{A}'$ bends toward the same horizontal direction near its ends with respect to $h$, the direction of the horizontal isotopies is well-defined.
We call the sequence of these horizontal isotopies the \textit{bigon-removing isotopy}, where the bigon-removing isotopy consists of one or more inductive steps.
After the bigon-removing isotopy, we have no bigon between $\partial D$ and $\bar{\nu}(\partial D)$.

Now we prove Claim \ref{claim-2} so that the bigon-removing isotopy would make sense.

\begin{claim}\label{claim-2}
There is an edge of the preimage for any new bigon such that it intersects an old bigon and  belongs to $\tilde{F}$ in each step of the bigon-removing isotopy.
\end{claim}

\begin{proofN}{Claim \ref{claim-2}}
Let us consider the intersection points $\partial D\cap \bar{\nu}(\partial D)$ before the bigon-removing isotopy.
If $p$ is a point of $\partial D\cap \bar{\nu}(\partial D)$, then there is a component $\mu$ of $\partial D\cap N(C_j)$ containing $p$ for some $1\leq j\leq 3$.
We say $p$ is the \textit{$k$-th from $\bar{F}$} if $p$ is the $k$-th from $\bar{C}_j$ among the points of $\partial D\cap \bar{\nu}(\partial D)$ in $\mu$.

Here, we can observe the follows (see Figure \ref{fig-outer-2}).
\begin{enumerate}
\item Each component  $\mu$ of $\partial D\cap N(C_j)$ or $\mu'$ of $\bar{\nu}(\partial D)\cap N(C_j)$ has $2|\partial D\cap N(C_j)|$ points of $\partial D\cap \bar{\nu}(\partial D)$ for $1\leq j \leq 3$.
Moreover, $\mu$ has the $k$-th point from $\bar{F}$ for every $1\leq k \leq 2|\partial D\cap N(C_j)|$.\label{observation-A1-1}
\item If we consider $N(C_j)$ as the union of the upper half and the lower half such that each half represents  one-time winding of $\mu'$ about $C_j$, where the lower half intersects $\bar{C}_j$, then the points of $\partial D\cap \bar{\nu}(\partial D)$ in the lower half are at most the $|\partial D\cap N(C_j)|$-th from $\bar{F}$ and the points in the upper half are at least the $(|\partial D\cap N(C_j)|+1)$-th from $\bar{F}$.
This means the points of $\partial D\cap \bar{\nu}(\partial D)$ in the upper half of $\mu'$ are at least the $(|\partial D\cap N(C_j)|+1)$-th from $\bar{F}$.
\item In summary, each of $\mu$ and $\mu'$ must have at least one point of $\partial D\cap \bar{\nu}(\partial D)$ that is at least the $(|\partial D\cap N(C_j)|+1)$-th from $\bar{F}$.\label{observation-A1}
\end{enumerate}

\begin{figure}
\includegraphics[width=6cm]{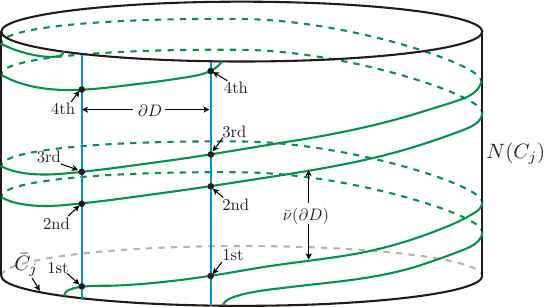}
\caption{$\partial D\cap \bar{\nu}(\partial D)$ in $N(C_j)$\label{fig-outer-2}}
\end{figure}

We will use an induction argument for the steps of the bigon-removing isotopy.

First, we consider the first-step of the bigon-removing isotopy.

Let $\bar{R}$ be the preimage of a new bigon appearing in the second step of the bigon-removing isotopy and $\partial\bar{R}=\alpha\cup\alpha'$ as in (A), (B) and (C) of Figure \ref{fig-gon}.
Since $\bar{R}$ is the preimage of a new bigon, there is a first-step bigon $R$ such that $R$ intersects $\alpha$ and $R\cap N(C_i)\neq \emptyset$ for $i=1$ or $3$ without loss of generality.
Here, every point of $\partial D \cap \bar{\nu}(\partial D)$ in $\mathrm{int}(\alpha)$ must be the 1st from $\bar{F}$ because the old bigons are the first-step bigons.

Let $\tilde{\alpha}$ be the edge of $R$ contained in $\alpha$.
It is sufficient to show that if we extend $\tilde{\alpha}$ in $\alpha$, then the extension must complete $\alpha$ without leaving $\tilde{F}$.
By the shape of the first-step bigons, $\tilde{\alpha}$ must belong to an element $\beta$ of $\mathcal{A}$ or $\mathcal{A}'$ and $\beta$ must have a point of $\partial D \cap \bar{\nu}(\partial D)$ which is at least the $(\min(|\partial D\cap N(C_i)|, |\partial D\cap N(C_2)|)+1)$-th from $\bar{F}$ in each of $N(C_i)$ and $N(C_2)$ by (\ref{observation-A1}).
Here, these intersection points are not the 1st from $\bar{F}$, i.e. they are not the points of $\partial D\cap \bar{\nu}(\partial D)$ contained in $\mathrm{int}(\alpha)$.
Therefore, the extension of $\tilde{\alpha}$ completes $\alpha$ without leaving $\tilde{F}$.
This completes the proof for the first-step of the bigon-removing isotopy.

Suppose there is an edge of the preimage for any new bigon appearing in the next step such that it intersects an old bigon and  belongs to $\tilde{F}$ in the $m$-th step of the bigon-removing isotopy for every $1\leq m\leq n$.
This means every bigon in the 2nd, $\cdots$, $(m+1)$-th step consists of two triangles in $N(C_j)$ and $N(C_2)$ for $j=1$ or $3$ respectively and a rectangle in $\bar{F}$ for $1\leq m \leq n$ by using the observations before Claim \ref{claim-2} inductively.
That is, any vertex of a bigon  in the $(m+1)$-th step for $1\leq m \leq n$ is the closest to $\bar{F}$ among the remaining points of $\partial D\cap \bar{\nu}(\partial D)$ in the corresponding component of $\partial D\cap N(C_j)$ at that step and therefore the vertices of the $(m+1)$-th step bigons are at most the $(m+1)$-th from $\bar{F}$.
If there is no new bigon appearing in the $(n+2)$-th step of the bigon-removing isotopy, then the induction hypothesis means we have completed the proof.
Hence, assume (i) there is a preimage $\bar{R}$ of a new bigon appearing in the $(n+2)$-th  step of the bigon-removing isotopy and let $\partial\bar{R}=\alpha\cup\alpha'$ as in (A), (B) and (C) of Figure \ref{fig-gon} and (ii) we are at the very moment when the $n$-th step ends.
Since $\bar{R}$ is the preimage of a new bigon, there is a bigon $R$ in the $(n+1)$-th step  such that $R$ intersects $\alpha$  without loss of generality and $R\cap N(C_i)\neq \emptyset$ for $i=1$ or $3$.
Here, every point of $\partial D \cap \bar{\nu}(\partial D)$ in $\mathrm{int}(\alpha)$ must be at most $k$-th from $\bar{F}$ for some  $1\leq k\leq (n+1)$ by the shape of the $(n+1)$-th step bigons and we assume $R$ has a vertex which is the $k$-th from $\bar{F}$ without loss of generality.

Recall that the direction of the bigon-removing isotopy is fixed and each  component of $\mathcal{A}'$ which was affected in the $m$-th step of the bigon-removing isotopy for $1\leq m \leq n$ passed through exactly one component of $\partial D\cap N(C_j)$ in that step if it intersects $N(C_j)$.
Moreover, every element of $\mathcal{A}'$ moves continuously intersecting only one of $\partial V-\partial B$ or $\partial W' - \partial B$ during the bigon-removing isotopy before the $(n+1)$-th step, i.e. the bigon-removing isotopy cannot loosen up each element of $\mathcal{A}'$ more than one time winding about $\tilde{C}_j$ right after the $n$-th step for $j=1,3$ and therefore this also holds for $j=2$.
This means for a component $\mu_j$ of $\partial D\cap N(C_j)$ for $1\leq j\leq 3$, if an element of $\mathcal{A}'$ once passed through $\mu_j$ in some step, then it cannot pass through $\mu_j$ twice when we consider  at most the $n$-th step.
Since the bigon-removing isotopy itself is also well defined for the $(n+1)$-th step by the shape of the $(n+1)$-th step bigons, the previous argument also holds even if we replace \textit{the $n$-th step} by \textit{the $(n+1)$-th step} which was not done yet.

Suppose $\mu_i$ is the component of $\partial D\cap N(C_i)$ intersecting the vertical edge of $R$ and assume the vertex of $R$ in $\mu_i$ is the $k_i$-th from $\bar{F}$.
Then $(k_i-1)$-elements of $\mathcal{A}'$ passed through $\mu_i$ removing each  point of $\partial D\cap \bar{\nu}(\partial D)$ from $\mu_i$ in the relevant step before the $(n+1)$-th step.
Since there is at least one element of $\mathcal{A}'$ that didn't pass through $\mu_i$ yet by the existence of the $(n+1)$-th step bigon $R$, the number of elements of $\mathcal{A}'$ that passed through $\mu_i$ before the $(n+1)$-th step is at most $|\bar{\nu}(\partial D)\cap N(C_i)|-1$.
Therefore, we get $k_i-1\leq|\bar{\nu}(\partial D)\cap N(C_i)|-1=|\partial D\cap N(C_i)|-1$, i.e. $k_i\leq |\partial D\cap N(C_i)|$.
Similarly, if we consider the component $\mu_2$ of $\partial D\cap N(C_2)$ intersecting the vertical edge of $R$ and assume the vertex of $R$ in $\mu_2$  is the $k_2$-th from $\bar{F}$, then we get  $k_2\leq|\partial D\cap N(C_2)|$.
But an element of $\mathcal{A}'$ passes through $\mu_i$ if and only if it passes through $\mu_2$  when we consider at most the $n$-th step of the bigon-removing isotopy because $\mu_i$ and $\mu_2$ belong to the same element of $\mathcal{A}$. 
This means $(k_i-1)=(k_2-1)$, i.e. $k_i=k_2=k$.
Therefore, we get $k\leq \min(|\partial D\cap N(C_i)|,|\partial D\cap N(C_2)|)$.

Let $\tilde{\alpha}$ be the edge of $R$ contained in $\alpha$.
By the shape of the $(n+1)$-th step bigons, $\tilde{\alpha}$ must belong to an element $\beta$ of $\mathcal{A}$ or $\mathcal{A}'$.
If we extend $\tilde{\alpha}$ in $\beta$ until it leaves $\tilde{F}$, then either (i) we meet a point of $\partial D\cap \bar{\nu}(\partial D)$ which is at least the $(\min(|\partial D\cap N(C_i)|,|\partial D\cap N(C_2)|)+1)$-th from $\bar{F}$ in each of $N(C_i)$ and $N(C_2)$ by (\ref{observation-A1}) in the start of the proof or (ii) a vertex of $R$ is at least the $(\min(|\partial D\cap N(C_i)|,|\partial D\cap N(C_2)|)+1)$-th from $\bar{F}$.
But the inequality
\[k <\min(|\partial D\cap N(C_i)|,|\partial D\cap N(C_2)|)+1\]
means any point of $\partial D\cap \bar{\nu}(\partial D)$ which is at least the $(\min(|\partial D\cap N(C_i)|,|\partial D\cap N(C_2)|)+1)$-th from $\bar{F}$ cannot belong to $\mathrm{int}(\alpha)$ by the assumption of $k$, i.e. only the former statement holds.
Therefore, the extension of $\tilde{\alpha}$ completes $\alpha$ without leaving $\tilde{F}$.
This completes the proof for the $(n+1)$-th step.

This completes the proof of Claim \ref{claim-2}.
\end{proofN}
 
From now on, we will prove $\bar{\nu}(\partial D)$ intersects $\partial D$ after the bigon-removing isotopy.

If there is an element of $\mathcal{A}'$ fixed during the bigon-removing isotopy, then it intersects $\partial D$ in each of $N(C_i)$ and $N(C_2)$ for $i=1$ or $3$ after the bigon-removing isotopy by Assumption \ref{assumption-5}.
Therefore, $\bar{\nu}(\partial D)\cap \partial D\neq\emptyset$ after the bigon-removing isotopy, leading to the result.

Hence, we assume every element of $\mathcal{A}'$ has been moved during the bigon-removing isotopy.
Let us consider the circle $\Gamma=\mathrm{cl}((\partial V\cup \partial W')- \partial B)$.
Assume $\mathcal{A}'$ is the one before the bigon-removing isotopy.
If we consider $\Gamma\cap\mathcal{A}'$ by tracing $\Gamma$ along the direction of the bigon-removing isotopy, then we get two sequences of elements of $\mathcal{A}'$, say $s_{V}$ and $s_{W'}$, such that each element of $s_{V}$ intersects $\partial V$ and each element of $s_{W'}$ intersects $\partial W'$, i.e. each element of $s_{V}$ connects $\tilde{C}_1$ and $\tilde{C}_2$ and each element of $s_{W}$ connects $\tilde{C}_3$ and $\tilde{C}_2$.
Hence, every element of $s_{W'}$ appears between the last element and the first element of $s_{V}$ with respect to the order of points of $\Gamma\cap\mathcal{A}'$ in $\Gamma$ and vise versa.
Since we have at least one bigon, $s_{V}\neq\emptyset$ or $s_{W'}\neq\emptyset$.

Suppose both $s_{V}$ and $s_{W'}$ consist of at most one element.
If there is the $n$-th step for $n\geq 2$ in the bigon-removing isotopy, then there must be a preliminary $4$-gon or $6$-gon $R'$ as in Figure \ref{fig-gon} in the $(n-1)$-th step such that the preimage containing $R'$ becomes a bigon in the $n$-th step.
But we can discard a $4$-gon of type (a) and a $6$-gon since 
each of them needs two elements of $s_{V}$ or $s_{W'}$ intersecting $R'$.
Suppose $R'$ is a $4$-gon of type (b).
In this case, there is one more element of $s_{V}$ or $s_{W'}$ near the ``long'' vertical edge of $R'$ as in (b) of Figure \ref{fig-gon} by Assumption \ref{assumption-3}, violating the assumption.
Hence, the bigon-removing isotopy consists of only one step and therefore every element of $\mathcal{A}'$ bounds a bigon with an element of $\mathcal{A}$ initially.
Let $\zeta$ be an element of $\mathcal{A}'$ and consider $|\zeta\cap (\partial D\cap N(C_2))|$.
Recall that $|\zeta\cap (\partial D\cap N(C_2))|\geq 2$ before the bigon-removing isotopy.
Since $|\zeta\cap (\partial D\cap N(C_2))|$ decreases by only one in the first step of the bigon-removing isotopy and there is no another step, we conclude $|\zeta\cap (\partial D\cap N(C_2))|\geq 1$ after the bigon-removing isotopy.
That is, $\bar{\nu}(\partial D)\cap \partial D\neq\emptyset$ after the bigon-removing isotopy, leading to the result.

Now assume one of $s_{V}$ and $s_{W'}$, say $s_{V}$, consists of at least two elements.
Let $\zeta$ and $\zeta_0$ be the last and the first elements of $s_{V}$ respectively.
Then $\zeta\cup\zeta_0$ divides $\bar{F}$ into a disk and an annulus in any step of the bigon-removing isotopy because both intersect $\partial V-\partial B$ during the bigon-removing isotopy, where the region from $\zeta_0$ to $\zeta$ in $\bar{F}$ is the disk and the region from $\zeta$ to $\zeta_0$ in $\bar{F}$ is the annulus with respect to the direction of the bigon-removing isotopy.

Suppose $\zeta$ is moved by the bigon-removing isotopy in the $n$-th step for $n\geq 2$, i.e. $\zeta$ bounds a bigon $R^n$ with an element of $\mathcal{A}$ in the $n$-th step.
Then there must be a preliminary $4$-gon or $6$-gon $R'$ as in Figure \ref{fig-gon} in the $(n-1)$-th step such that the preimage containing $R'$ becomes $R^n$ after the $(n-1)$-th step.
If two elements of $\mathcal{A}'$ intersect $R'$, then we call them ``\textit{the 1st element intersecting $R'$}'' and ``\textit{the 2nd element intersecting $R'$}'' with respect to the direction of the bigon-removing isotopy.
Suppose $R'$ is a $4$-gon of type (a) or a $6$-gon.
If we observe (a) and (c) of Figure \ref{fig-gon}, then we can see the 1st element of $\mathcal{A}'$ intersecting $R'$ becomes the element of $\mathcal{A}'$ intersecting $R^n$ after the $(n-1)$-th step.
This means $\zeta$ must be the 1st element intersecting $R'$ and therefore $\zeta_0$ is the 2nd element intersecting $R'$, violating the assumption that the region from $\zeta$ to $\zeta_0$ in $\bar{F}$ is an annulus.
Suppose $R'$ is a $4$-gon of type (b).
In this case, there is another element of $s_{V}$ near the long vertical edge of $R'$ as we saw previously.
Since this element appears next to $\zeta$ with respect to the direction of the bigon-removing isotopy, it should be $\zeta_0$.
But the region from $\zeta$ to $\zeta_0$ in $\bar{F}$ is a disk as in (b) of Figure \ref{fig-gon}, not an annulus, violating the previous paragraph.
Therefore, $\zeta$ cannot be moved by the bigon-removing isotopy in the $n$-th step for $n\geq 2$.

Hence, $\zeta$ bounds a bigon with a subarc of $\mathcal{A}$ before the bigon-removing isotopy and $\zeta$ is fixed after the first step.
This means $|\zeta\cap (\partial D\cap N(C_2))|$ decreases by only one during the bigon-removing isotopy, i.e. $|\zeta\cap (\partial D\cap N(C_2))|\geq 1$ after the bigon-removing isotopy similarly as the previous case.
Therefore, $\bar{\nu}(\partial D)\cap \partial D\neq\emptyset$ after the bigon-removing isotopy, leading to the result.

This completes the proof of Claim \ref{claim-1}.
\end{proofN}

\begin{claim}\label{claim-3}
The geometric intersection number $i(\partial D,\bar{\nu}^k(\partial D))= i(\partial D,\bar{\nu}(\partial D)) + (k-1)c$ for some $c>0$ and it is realized by a sequence of isotopies of $\nu^k(D)$ in $\W$ for $k\geq 1$, where $c, k\in \mathbb{N}$.
\end{claim}

\begin{proofN}{Claim \ref{claim-3}}
If $k=1$, then there is nothing to prove.
Let us do $\nu$ twice.

We will do the bigon-removing isotopy similarly as in Claim \ref{claim-1}.
Let $\mu$ be a component of $\partial D\cap N(C_i)$ for $1\leq i \leq 3$ and consider the points $\mu\cap\bar{\nu}^2(\partial D)$ in $\mu$.
We can see $\mu\cap \bar{\nu}^2(\partial D)= 4|\partial D\cap N(C_i)|$ before the bigon-removing isotopy.
If we consider the proof of Claim \ref{claim-2}, then at most $|\partial D\cap N(C_i)|$ points of  $\mu\cap \bar{\nu}^2(\partial D)$ in $\mu$ disappeared after the bigon-removing isotopy and the remaining points in $\mu$ are more farther from $\bar{F}$ than the disappeared points in $\mu$.
That is, the first half of $\mu\cap \bar{\nu}^2(\partial D)$ in $\mu$ is affected and the other half is unaffected.
Hence, if we observe every component of $\partial D\cap (\cup_{j=1}^3 N(C_j))$, then the sum of the numbers of intersection points corresponding to the first half is the geometric intersection number $i(\partial D,\bar{\nu}(\partial D))$.
Moreover, that corresponding to the other half is $c=|\partial D \cap \bar{\nu}(\partial D)|$ when $\bar{\nu}(\partial D)$ is the one before the bigon-removing isotopy. 
Hence, we get
$$i(\partial D,\bar{\nu}^2(\partial D))= i(\partial D,\bar{\nu}(\partial D)) + c.$$
We can generalize this argument to $\nu^k$  and then we get
$$i(\partial D,\bar{\nu}^k(\partial D))= i(\partial D,\bar{\nu}(\partial D)) + (k-1)c.$$
This completes the proof of Claim \ref{claim-3}.
\end{proofN}

If we consider the family of disks $\{\nu^{k}(D)\}_{k=0}^\infty$, then $i(\partial D,\bar{\nu}^{k_1}(\partial D))\neq i(\partial D,\bar{\nu}^{k_2}(\partial D))$ for  $k_1\neq k_2$ by Claim \ref{claim-3} and these geometric intersection numbers are realized by the bigon-removing isotopies of $\nu^{k_1}(D)$ and $\nu^{k_2}(D)$ in $\W$.
This means two disks $\nu^{k_1}(D)$ and $\nu^{k_2}(D)$ are not isotopic in $\W$ and therefore they represent different vertices in $\D(F)$.
Hence, the orbit $Mod(M,F).[D]$ contains the infinite set $\{[\nu^{k}(D)]\}_{k=0}^\infty$.

This completes the proof.
\end{proof}

\begin{lemma}\label{lemma-disconnected-infinite-II}
Assume $M$ and $F$ as in Lemma \ref{lemma-disconnected-infinite-I}.
Suppose $(V,W)$ is a weak reducing pair of $(\V,\W;F)$ such that $V$ and $W$ are non-separating in $\V$ and $\W$ respectively and $\partial V\cup\partial W$ is separating in $F$.
If $D$ is a compressing disk of $F$ such that $\partial D$ intersects at least one of $\partial V$ and $\partial W$ up to isotopy in $F$, then the orbit $Mod(M,F).[D]$ of the element $[D]\in \D(F)$ consists of infinitely many elements.
\end{lemma}

\begin{proof}
We will use the similar arguments in the proof of Lemma \ref{lemma-disconnected-infinite-I}.
Without loss of generality, assume $D\subset \W$ and $\partial D$ intersects $\partial V$ in $F$ up to isotopy.

By Lemma \ref{lemma-2-8}, $\partial V\cup\partial W$ cuts off two twice-punctured surfaces of genus at least one $F_1'$ and $F_2'$ from $F$.
Hence, we can find a simple path $\gamma$ connecting $\partial V$ and $\partial W$ in $F_2'$ such that $\mathrm{int}(\gamma)$ misses $\partial V\cup\partial W$.
If we drag $\partial W$ along $\gamma$ and push it more into $F_1'$, then we get a disk $W'$ which intersects $V$ transversely in exactly two points, i.e. a subarc of $\partial W'$ bounds a bigon $B$ with a subarc of $\partial V$ in $F_1'$.

We can see $\partial V\cup\partial W'$ divides $F$ into three regions $B$, $F_1$ and $F_2$ such that
\begin{enumerate}
	\item $\mathrm{int}(F_1)$ is an open twice-punctured surface of genus at least one and $F_1\cap B=\partial B$,  and
	\item $F_2$ is a once-punctured surface of genus at least one such that $F_2\cap B$ consists of two points (see Figure \ref{fig-torus-twist-ns}).
\end{enumerate}
\begin{figure}
\includegraphics[width=9cm]{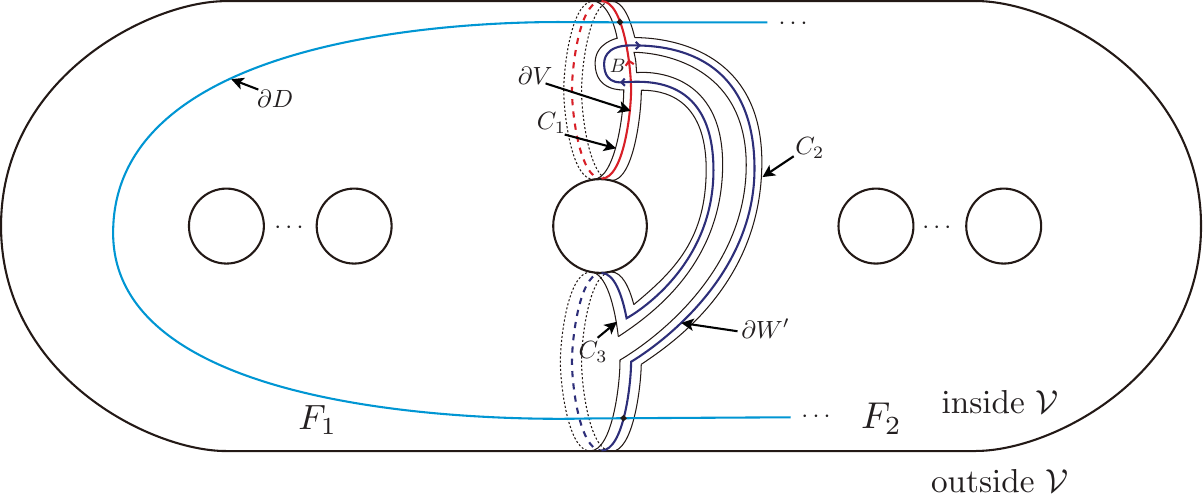}
\caption{the three curves $C_1$, $C_2$, $C_3$\label{fig-torus-twist-ns}}
\end{figure}

Let $\mathcal{T}$ be the solid torus $N(V\cup W')$ and $\tau$ a torus twist defined by $\mathcal{T}$.
Then $\mathcal{T}\cap F$ is a four-punctured sphere $S$ and we can assume the induced map $\bar{\tau}$ in $Mod(F)$ from $\tau$ consists of three Dehn twists about the three curves $C_1$, $\cdots$, $C_3$ in $\mathrm{int}(N(S)-S)$ after discarding the trivial Dehn twist about the curve parallel to $\partial B$ (see Figure \ref{fig-torus-twist-ns}).
Similarly as in the proof of Lemma \ref{lemma-disconnected-infinite-I}, we will use the naming of the four curves and the direction of $\tau$ as in Figure \ref{fig-def-torus-twist} by substituting $V$ and $W'$ for $V$ and $W$ respectively.
Hence, $\bar{\tau}$ is the union of three non-trivial Dehn twists in the disjoint annuli $N(C_1)$, $N(C_2)$ and $N(C_3)$, where $N(C_1)$ and $N(C_3)$ are contained in $\mathrm{int}(F_1)$ and $N(C_2)$ is contained in $\mathrm{int}(F_2)$.
Here, we give an $I$-fibration to each $N(C_i)$ as we did in Definition \ref{definition-torus-twist} for $1\leq i \leq 3$.

Define $\bar{F}$, $\tilde{F}$, $\tilde{C}_i$ and $\bar{C}_i$ as in the proof of Lemma \ref{lemma-disconnected-infinite-I}.
Let $\tilde{F}_1$ and $\tilde{F}_2$ be the closure of the twice-punctured  component and the once-punctured  component of $F-(\cup_{j=1}^3 \tilde{C}_j)$ respectively.

Then we isotope $D$ as we did before Claim \ref{claim-1} and get Assumption \ref{assumption-1}, Assumption \ref{assumption-2}, Assumption \ref{assumption-3}, Assumption \ref{assumption-4} and Assumption \ref{assumption-5}.

We will prove $\partial D$ must intersect at least one of $N(C_1)$ and $N(C_3)$ up to isotopy.
Suppose $\partial D$ misses $N(C_1)\cup N(C_3)$ after an isotopy of $D$ in $\W$.
We can see $\mathrm{cl}((F_1\cup B)-(N(C_1)\cup N(C_3)))$ consists of a twice-punctured surface $F'$ and a pair of pants $F''$.
Since $\partial D$ intersects $\partial V$ up to isotopy, $\partial D$ cannot belong to $F'$ and $\partial D\cap F''$ must have at least one properly embedded essential arc in $F''$.
Here, every properly embedded essential arc of $\partial D\cap F''$  separates the two components of $\partial F''$  intersecting $N(C_1)\cup N(C_3)$.
This means such arcs are parallel in $F''$.
Hence, if we isotope $D$ so that (i) this isotopy removes every inessential arcs of $\partial D\cap F''$ from $F''$ by a standard outermost disk argument and then (ii) makes the remaining parallel essential arcs of $\partial D\cap F''$ not intersect $\partial V$, where the latter isotopy is represented by pushing the parallel arcs into the region in $F''$ missing $\partial V$, then  this violates the assumption that $\partial D$ intersects $\partial V$ up to isotopy.

Let us do the torus twist $\tau$ twice and let $\nu$ be $\tau^2$.
As in the proof of Lemma \ref{lemma-disconnected-infinite-I}, we prepare a disjoint parallel copy of $D$ in $\W$ near $D$ and consider the image of this copy of $\nu$ as $\nu(D)$.
Let $\bar{\nu}$ be the induced map in $Mod(F)$ from $\nu$ (see Figure \ref{fig-high-twists-ns} for $\bar{\nu}(\partial D)$).

\begin{figure}
\includegraphics[width=9cm]{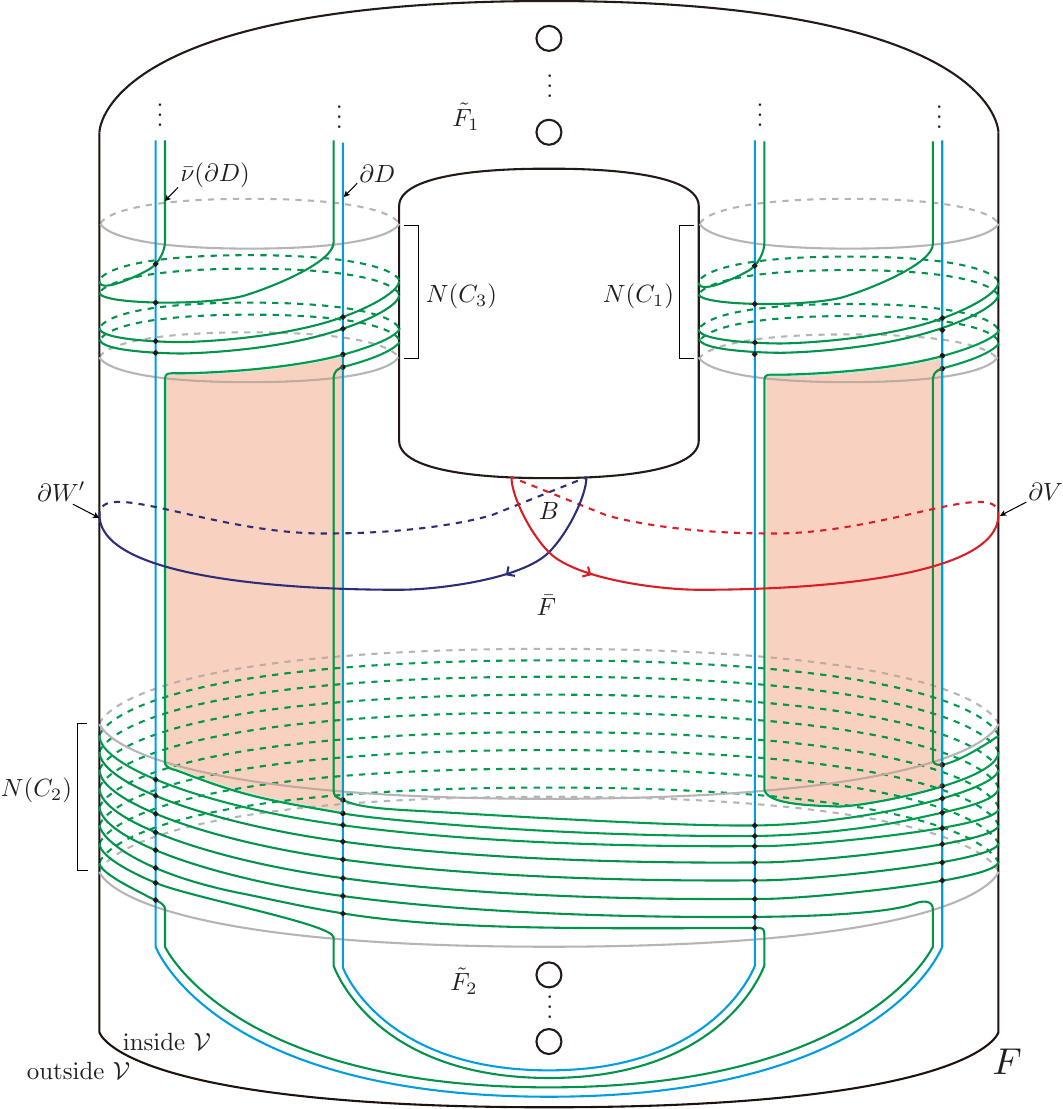}
\caption{the visualization of the three Dehn twists \label{fig-high-twists-ns}}
\end{figure}

\begin{claim}\label{claim-4}
The geometric intersection number $i(\partial D,\bar{\nu}(\partial D))>0$ and it is realized by a sequence of isotopies of $\nu(D)$ in $\W$.
\end{claim}

\begin{proofN}{Claim \ref{claim-4}}
Since $\partial D$ must intersect at least one of $N(C_1)$ and $N(C_3)$ up to isotopy, $\bar{\nu}(\partial D)$ intersects $\partial D$ in $N(C_1)$ or $N(C_3)$  by Assumption \ref{assumption-5}. 
Hence, if there is no bigon between $\partial D$ and $\bar{\nu}(\partial D)$, then the bigon criterion leads to the result.

Suppose there is a bigon $R$ between $\partial D$ and $\bar{\nu}(\partial D)$.
Then the possibilities are as follows by the similar arguments in the proof of Claim \ref{claim-1}.\begin{enumerate}
\item $R\subset\tilde{F}_1\cup N(C_i)$ for some $i=1$ or $3$.\label{px1}
\item $R\subset\tilde{F}_2\cup N(C_2)$.\label{px2}
\item $R\subset\bar{F}\cup N(C_i)$ for some $1\leq i \leq 3$.\label{px3}
\item $R\subset N(C_1)\cup\tilde{F}_1\cup N(C_3)$, $R\cap N(C_1)\neq\emptyset$, and $R\cap N(C_3)\neq\emptyset$.\label{px33}
\item $R\subset N(C_1)\cup\bar{F}\cup N(C_3)$, $R\cap N(C_1)\neq\emptyset$, and $R\cap N(C_3)\neq\emptyset$.\label{px4}
\item $R\subset N(C_i)\cup\bar{F}\cup N(C_2)$, $R\cap N(C_i)\neq\emptyset$, and $R\cap N(C_2)\neq\emptyset$ for $i=1$ or $3$.\label{px5}
\end{enumerate}
Here, $\bar{\nu}(\partial D)$ turns right in the non-separating annuli $N(C_1)$ and $N(C_3)$, and turns left in the separating annulus $N(C_2)$.
If we use the similar arguments in the proof of Claim \ref{claim-1}, then the only possible case is (\ref{px5}).
Hence, we can do the bigon-removing isotopy similarly as in the proof of Claim \ref{claim-1} and therefore there must be at least one intersection point between $\partial D$ and $\bar{\nu}(\partial D )$ in $N(C_2)$ after the bigon-removing isotopy.

This completes the proof of Claim \ref{claim-4}.
\end{proofN}

Therefore, Claim \ref{claim-4} leads to the same result as in Claim \ref{claim-3}.
Hence, the orbit $Mod(M,F).[D]$ contains the infinite set $\{[\nu^{k}(D)]\}_{k=0}^\infty$ similarly as in Lemma \ref{lemma-disconnected-infinite-I}.

This completes the proof.
\end{proof}

Finally, we reach Theorem \ref{theorem-mod-invariant-3-main}.

\begin{theorem}\label{theorem-mod-invariant-3-main}
Let $(\V,\W;F)$ be a weakly reducible, unstabilized Heegaard splitting of genus at least three in an orientable, irreducible $3$-manifold $M$.
If $F$ is topologically minimal and its topological index is two, then  the orbit $Mod(M,F).[D]$ of any element $[D]\in\D(F)$ consists of infinitely many elements.
This means if there is an element of $\D(F)$ having finite orbit, then (i) $F$ is not topologically minimal or (ii) the topological index of $F$ is at least three if $F$ is topologically minimal.
\end{theorem}

\begin{proof}
Since $F$ is topologically minimal and its topological index is two, $\DVW(F)$ is disconnected (see Definition \ref{definition-top-minimal}), i.e. there are at least two components $\mathcal{C}_1$ and $\mathcal{C}_2$ of $\DVW(F)$.
By Lemma \ref{lemma-2-9}, there exists a weak reducing pair $(V_i,W_i)\subset \mathcal{C}_i$ such that (i) $V_i$ and $W_i$ are separating in $\V$ and $\W$ respectively or (ii) $V_i$ and $W_i$ are non-separating in $\V$ and $\W$ respectively and $\partial V_i\cup\partial W_i$ is separating in $F$ for $i=1,2$.

Let $D$ be a compressing disk of $F$ such that $[D]\in\D(F)$ does not belong to $\mathcal{C}_1$.
Without loss of generality, assume $D\subset\V$.
Then $\partial D$ must intersect $\partial W_1$ up to isotopy in $F$ otherwise the isotopy of $\partial D$ in $F$ making it miss $\partial W_1$ can be realized by an isotopy of $D$ in $\V$ and therefore there would be a weak reducing pair $(D,W_1)$, violating the assumption that $[D]$ does not belong to $\mathcal{C}_1$.
Therefore, if we apply Lemma \ref{lemma-disconnected-infinite-I} or  Lemma \ref{lemma-disconnected-infinite-II} to the weak reducing pair $(V_1,W_1)$ and the disk $D$, then we conclude the orbit $Mod(M,F).[D]$ consists of infinitely many elements.

By the symmetric argument, if $[D']\in\D(F)$ is an element not belonging to $\mathcal{C}_2$, then the orbit $Mod(M,F).[D']$ also consists of infinitely many elements.
Because every element of $\D(F)$ does not belong to at least one of $\mathcal{C}_1$ and $\mathcal{C}_2$, we conclude the orbit of any element of $\D(F)$ consists of infinitely many elements.

This completes the proof of the first statement.

If $F$ is topologically minimal and weakly reducible, then the topological index is at least two and therefore the second statement comes from the first statement directly.

This completes the proof.
\end{proof}

\section{The proof of Theorem \ref{theorem-mod-invariant-4}\label{section4}}

In this section, we will prove Theorem \ref{theorem-mod-invariant-4}.

Before proving Theorem \ref{theorem-mod-invariant-4}, we introduce the following lemma (see \cite{10} for the definition of a $\V$- or $\W$-\textit{facial cluster}).

\begin{lemma}[the proof of Theorem 1.1 of \cite{JungsooKim2014}]\label{lemma-agt}
Let $(\V,\W;F)$ be a weakly reducible, unstabilized Heegaard splitting of genus three in an orientable, irreducible $3$-manifold such that $F$ is not topologically minimial.
Then $\DVW(F)$ is one of the following five types.
\begin{enumerate}
\item \underline{Case (a) of the proof of Lemma 3.5 of \cite{JungsooKim2014}:}

$\DVW(F)=\bigcup_{V,W} \Sigma_{VW}$ for all possible $V\subset\V$ and $W\subset\W$, where $\Sigma_{VW}$ is a $3$-simplex of the form $\{V,\bar{V},\bar{W},W\}$ for a fixed weak reducing pair $(\bar{V},\bar{W})$.
We say $\DVW(F)$ is of ``\textit{type (a)}''.
\item \underline{Case (b) of the proof of Lemma 3.5 of \cite{JungsooKim2014}:}

$\DVW(F)$ is a $\W$-facial cluster.
We say $\DVW(F)$ is of ``\textit{type (b)-$\W$}''.
\item \underline{the symmetric case of Case (b) of the proof of Lemma 3.5 of \cite{JungsooKim2014}:}

$\DVW(F)$ is a $\V$-facial cluster.
We say $\DVW(F)$ is of ``\textit{type (b)-$\V$}''.
\item \underline{Case (c) of the proof of Lemma 3.5 of \cite{JungsooKim2014}:}  

$\DVW(F)$ is a weak reducing pair such that each of them cuts off $(\mathrm{torus})\times I$ from the relevant compression body.
We say $\DVW(F)$ is of ``\textit{type (c)}''.
\item \underline{the proof of Lemma 3.7 of \cite{JungsooKim2014}:}  

$\DVW(F)$ is a weak reducing pair such that both disks are non-separating and the union of their boundaries is separating in $F$.
We say $\DVW(F)$ is of ``\textit{type (d)}''.
\end{enumerate}
For each of the first three cases, there is a uniquely determined weak reducing pair $(\bar{V},\bar{W})$ having the following property.
\begin{enumerate}
\item Type (a): $\bar{V}$ and $\bar{W}$ are non-separating in $\V$ and $\W$ respectively.
(In this case, $\partial \bar{V}\cup\partial \bar{W}$ is non-separating in $F$.)
\item Type (b)-$\W$: $\bar{V}$ cuts off $(\text{torus})\times I$ from $\V$ and $\bar{W}$ is non-separating in $\W$.
\item Type (b)-$\V$: $\bar{W}$ cuts off $(\text{torus})\times I$ from $\W$ and $\bar{V}$ is non-separating in $\V$.
\end{enumerate}
In any case, another disk of $\DVW(F)\cap\DV(F)$ other than $\bar{V}$ is a band-sum of two parallel copies of $\bar{V}$ in $\V$ if it exists.
Symmetrically, another disk of $\DVW(F)\cap\DW(F)$ other than $\bar{W}$ is a band-sum of two parallel copies of $\bar{W}$ in $\W$ if it exists.
This means another disk in $\DVW(F)$ other than $\bar{V}$ and $\bar{W}$ cuts off a solid torus from the relevant compression body if it exists.
\end{lemma}

We can refer to Figure \ref{fig-1} for the shapes of $\DVW(F)$ when $F$ is not topologically minimal.
\begin{figure}
\includegraphics[width=12cm]{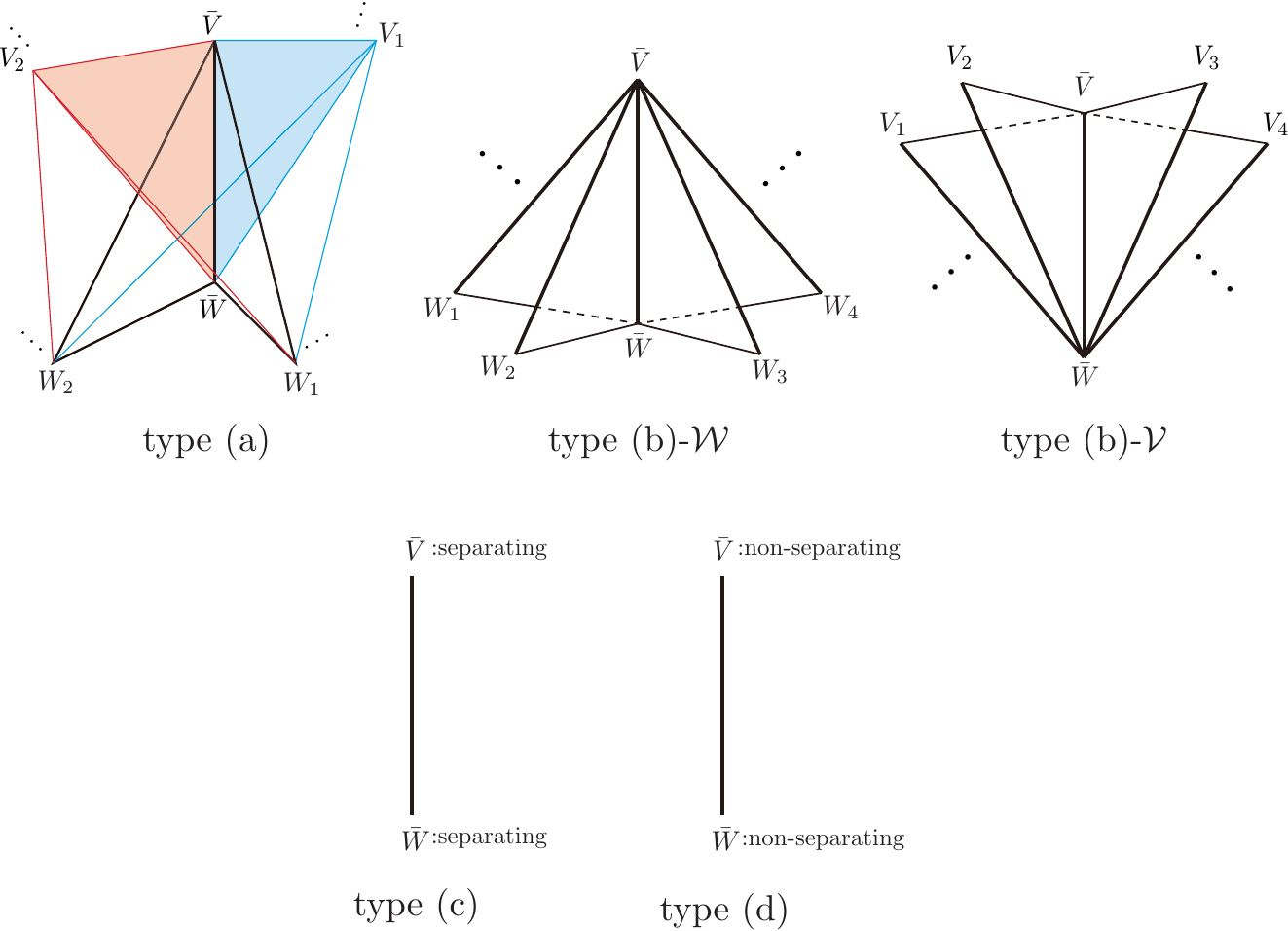}
\caption{The shapes of $\DVW(F)$ \label{fig-1}}
\end{figure}

\begin{lemma}\label{lemma-agt-ii} 
Assume $M$, $F$ and $(\bar{V},\bar{W})$ as in Lemma \ref{lemma-agt}.
Then either (i) $Mod(M,F).[\bar{V}]=\{[\bar{V}]\}$ and $Mod(M,F).[\bar{W}]=\{[\bar{W}]\}$ or  (ii) $Mod(M,F).[\bar{V}]=Mod(M,F).[\bar{W}]=\{[\bar{V}],[\bar{W}]\}$.
\end{lemma}

\begin{proof}
Let $[f]$ be an element of $Mod(M,F)$.
Since $\bar{V}$ does not cut off a solid torus from $\V$ by Lemma \ref{lemma-agt} (recall that the genus of $F$ is three), $f(\bar{V})$ also satisfies this property in the relevant compression body and the symmetric argument also holds for $f(\bar{W})$.
But there is only one weak reducing pair in $\DVW(F)$ consisting of such two disks by Lemma \ref{lemma-agt}, i.e. the weak reducing pair determined by $\{f(\bar{V}),f(\bar{W})\}$ is $(\bar{V},\bar{W})$ itself in $\D(F)$.
Hence, $Mod(M,F).[\bar{V}]=\{[\bar{V}]\}$ or $\{[\bar{V}],[\bar{W}]\}$ by considering $[\mathrm{id}_M]\in Mod(M,F)$ and the symmetric argument also holds for $Mod(M,F).[\bar{W}]$.

If $[f].[\bar{V}]=[\bar{V}]$ for every $[f]\in Mod(M,F)$,
then $[f].[\bar{W}]=[\bar{W}]$ for every $[f]\in Mod(M,F)$, leading to (i).
If $[g].[\bar{V}]=[\bar{W}]$ for some $[g]\in Mod(M,F)$, then $[g].[\bar{W}]=[\bar{V}]$, leading to (ii).

This completes the proof.
\end{proof}

Finally, we reach Theorem \ref{theorem-mod-invariant-4-main}.

\begin{theorem}\label{theorem-mod-invariant-4-main}
Let $(\V,\W;F)$ be a weakly reducible, unstabilized Heegaard splitting of genus three in an orientable, irreducible $3$-manifold $M$.
If $F$ is topologically minimal, then the orbit $Mod(M,F).[D]$ of any element of $[D]\in\D(F)$ consists of infinitely many elements.
Moreover, if $F$ is not topologically minimal, then there exist exactly two elements of $\D(F)$ having finite orbits, where either (i) two singleton sets consisting of each element are the only finite orbits or (ii) the set consisting of the two elements is the only finite orbit.
\end{theorem}

\begin{proof}
Suppose $F$ is topologically minimal.
In \cite{JungsooKim2014}, the author proved the topological index of $F$ must be two.
Therefore, the first statement comes from Theorem \ref{theorem-mod-invariant-3} directly.

Suppose $F$ is not topologically minimal.
In this case, $\DVW(F)$ is a component $\mathcal{C}$ as in Lemma \ref{lemma-agt}.

If $\mathcal{C}$ is of type (c) or type (d), then $\mathcal{C}$ consists of a weak reducing pair such that 
(1) both disks are separating or (2) both disks are non-separating and the union of their boundaries is separating in $F$.
Let $D$ be a compressing disk of $F$ such that $[D]\in\D(F)$ does not belong to $\mathcal{C}$ and assume $D\subset \V$ without loss of generality.
Then it must intersect the disk $\mathcal{C}\cap \DW(F)$ up to isotopy otherwise there would be another weak reducing pair other than $\mathcal{C}$.
This means the orbit $Mod(M,F).[D]$ consists of infinitely many elements by applying Lemma \ref{lemma-disconnected-infinite-I} or  Lemma \ref{lemma-disconnected-infinite-II} to the weak reducing pair $\mathcal{C}$ and the disk $D$.
Since every element of $Mod(M,F)$ sends a weak reducing pair into a weak reducing pair and there is only one weak reducing pair in $\mathcal{C}$, $Mod(M,F)$ preserves the weak reducing pair $\mathcal{C}$.
Hence, if we use the argument in the proof of Lemma \ref{lemma-agt-ii}, then one of (i) and (ii) holds, leading to the result.

Therefore, assume $\mathcal{C}$ is of type (a), type (b)-$\W$ or type (b)-$\V$.
If we consider Lemma \ref{lemma-agt-ii}, then the orbits corresponding to $\bar{V}$ and $\bar{W}$ satisfy (i) or (ii).
Hence, it is sufficient to show that every element of $\D(F)$ other than $\bar{V}$ and $\bar{W}$ has infinite orbit.

\Case{a} $\mathcal{C}$ is of type (b)-$\W$.

In this case, $\mathcal{C}$ is a $\W$-facial cluster, i.e. $\mathcal{C}\cap\DV(F)=\{\bar{V}\}$ and $(\mathcal{C}\cap\DW(F))-\{\bar{W}\}\neq \emptyset$.
Choose a vertex $\tilde{W}$ of $\mathcal{C}$ other than $\bar{V}$ and $\bar{W}$.
Here, $(\bar{V},\tilde{W})$ is a weak reducing pair consisting of separating disks by Lemma \ref{lemma-agt}.
Let us consider another vertex $W$ in $\mathcal{C}$ other than $\bar{V}$, $\bar{W}$ and $\tilde{W}$ (because the number of vertices of $\mathcal{C}$ is infinite (see Lemma 2.15 of \cite{10}), such $W$ exists).
Then both $\tilde{W}$ and $W$ are band-sums of two parallel copies of $\bar{W}$ in $\W$ by Lemma \ref{lemma-agt} and $\partial\tilde{W}$ and $\partial W$ are non-isotopic in $F$ by the assumption that $\tilde{W}\neq W$ in $\DW(F)$.
This means $\partial W\cap\partial \tilde{W}\neq\emptyset$ up to isotopy.
Hence, if we apply Lemma \ref{lemma-disconnected-infinite-I} to the weak reducing pair $(\bar{V},\tilde{W})$ and the disk $W$, then we conclude the orbit $Mod(M,F).[W]$ consists of infinitely many elements.
If we use the previous argument again by changing the roles of $\tilde{W}$ and $W$, then the orbit $Mod(M,F).[\tilde{W}]$ also consists of infinitely many elements, leading to the result.

\Case{b} $\mathcal{C}$ is of type (b)-$\V$.

In this case, we can use the symmetric argument of Case a.

\Case{c} $\mathcal{C}$ is of type (a).

In this case, $\mathcal{C}=\bigcup_{V,W} \Sigma_{VW}$ for all possible $V\subset\V$ and $W\subset\W$, where $\Sigma_{VW}$ is a $3$-simplex of the form $\{V,\bar{V},\bar{W},W\}$ for a fixed weak reducing pair $(\bar{V},\bar{W})$ by Lemma \ref{lemma-agt}.

\begin{claim}\label{claim-6}
For any $3$-simplex $\Sigma_{VW}$ in $\mathcal{C}$, there is a $3$-simplex $\Sigma_{\tilde{V}\tilde{W}}$ such that $\Sigma_{VW}\cap \Sigma_{\tilde{V}\tilde{W}}=\{\bar{V},\bar{W}\}$.
\end{claim}

\begin{proofN}{Claim \ref{claim-6}}
If we compress $F$ along $\bar{V}$ and $\bar{W}$, then we get the torus $F_{\bar{V}\bar{W}}$ because $\bar{V}$ and $\bar{W}$ are non-separating in $\V$ and $\W$ respectively and $\partial \bar{V}\cup\partial \bar{W}$ is non-separating in $F$ by Lemma \ref{lemma-agt}.
Hence, there are the scars of $\bar{V}$, say $\bar{V}_1$ and $\bar{V}_2$, and the scars of $\bar{W}$, say $\bar{W}_1$ and $\bar{W}_2$, in $F_{\bar{V}\bar{W}}$.
Since $V$ and $W$ are band sums of two parallel copies of $\bar{V}$ and $\bar{W}$ in $\V$ and $\W$ respectively by Lemma \ref{lemma-agt} and $V\cap W=\emptyset$, we can find simple arcs representing these band sums in $F_{\bar{V}\bar{W}}$, say $\alpha_{\bar{V}}$ and $\alpha_{\bar{W}}$, such that $\alpha_{\bar{V}}$ connects $\bar{V}_1$ and $\bar{V}_2$, $\alpha_{\bar{W}}$ connects $\bar{W}_1$ and $\bar{W}_2$, and $\alpha_{\bar{V}}\cap\alpha_{\bar{W}}=\emptyset$.
Let $F_{\bar{V}\bar{W}}'$ be the four-punctured torus $\mathrm{cl}(F_{\bar{V}\bar{W}}-(\bar{V}_1\cup\bar{V}_2\cup\bar{W}_1\cup\bar{W}_2))$.
Since $\mathrm{int}(F_{\bar{V}\bar{W}}'-(\alpha_{\bar{V}}\cup\alpha_{\bar{W}}))$ is an open twice-punctured torus, we can find a simple path $\gamma$ from $\bar{V}_1$ to $\bar{W}_1$ in $F_{\bar{V}\bar{W}}'-(\alpha_{\bar{V}}\cup\alpha_{\bar{W}})$ such that $\mathrm{int}(\gamma)$ belongs to $\mathrm{int}(F_{\bar{V}\bar{W}}')$.
Let us consider a path $\beta_{\bar{V}}$ from $\bar{V}_1$ to $\bar{V}_2$ missing $\gamma\cup\alpha_{\bar{V}}$ such that 
\begin{enumerate}
\item $\beta_{\bar{V}}$ starts from $\bar{V}_1$, where the starting point is in a small neighborhood of $\gamma$,
\item proceeds along a parallel copy of $\gamma$ until it reaches a small neighborhood of $\bar{W}_1$, say $N(\bar{W}_1)$,
\item turns around along $\partial N(\bar{W}_1)$ so that it would wrap around $\bar{W}_1$ nearly one time without intersecting $\gamma$,
\item returns to a small neighborhood of $\bar{V}_1$, say $N(\bar{V}_1)$, along another parallel copy of $\gamma$, 
\item turns around along $\partial N(\bar{V}_1)$ until it reaches a small neighborhood of $\alpha_{\bar{V}}$ without intersecting $\gamma$, and
\item proceeds along a parallel copy of $\alpha_{\bar{V}}$ and finally reaches $\bar{V}_2$.
\end{enumerate}
We can refer to Figure \ref{fig-four-curves}.
\begin{figure}
\includegraphics[width=8cm]{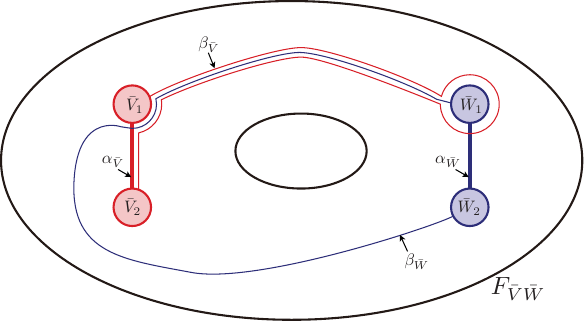}
\caption{obtaining $\tilde{V}$ and $\tilde{W}$ \label{fig-four-curves}}
\end{figure}
Then $\bar{V}_1\cup\alpha_{\bar{V}}\cup\bar{V}_2\cup \beta_{\bar{V}}$ separates $\bar{W}_1$ and $\bar{W}_2$ in $F_{\bar{V}\bar{W}}$ but $\alpha_{\bar{W}}$ connects $\bar{W}_1$ and $\bar{W}_2$ in $F_{\bar{V}\bar{W}}$.
Since $\alpha_{\bar{W}}$ misses $\alpha_{\bar{V}}$, we conclude 
$\beta_{\bar{V}}\cap\alpha_{\bar{W}}\neq\emptyset$ in $F_{\bar{V}\bar{W}}'$  up to isotopy.
Moreover, we can assume $\beta_{\bar{V}}$ intersects $\alpha_{\bar{W}}$ transversely in exactly one point by the construction of $\beta_{\bar{V}}$.

Here, $\mathrm{int}(F_{\bar{V}\bar{W}}'-\beta_{\bar{V}})$ is an open thrice-punctured torus and therefore we can find a simple path $\beta_{\bar{W}}$ from $\bar{W}_2$ to $\bar{W}_1$ in $F_{\bar{V}\bar{W}}'-\beta_{\bar{V}}$ such that $\mathrm{int}(\beta_{\bar{W}})$ belongs to $\mathrm{int}(F_{\bar{V}\bar{W}}'-\beta_{\bar{V}})$.
Since $\bar{V}_1\cup\alpha_{\bar{V}}\cup\bar{V}_2\cup \beta_{\bar{V}}$ separates $\bar{W}_1$ and $\bar{W}_2$ in $F_{\bar{V}\bar{W}}$ and $\beta_{\bar{W}}$ misses $\beta_{\bar{V}}$, $\beta_{\bar{W}}\cap\alpha_{\bar{V}}\neq\emptyset$ in $F_{\bar{V}\bar{W}}'$ up to isotopy.
Here, we can assume $\beta_{\bar{W}}$ intersects $\alpha_{\bar{V}}$ transversely in exactly one point.

Let $\tilde{V}$ be the band sum of two parallel copies of $\bar{V}$ in $\V$ realized by $\beta_{\bar{V}}$ and $\tilde{W}$ be the band sum of two parallel copies of $\bar{W}$ in $\W$ realized by $\beta_{\bar{W}}$.
By the assumptions of $\beta_{\bar{V}}\cap\alpha_{\bar{W}}$ and $\beta_{\bar{W}}\cap\alpha_{\bar{V}}$, we get $\tilde{V}\cap W\neq \emptyset$ and $\tilde{W}\cap V\neq \emptyset$ up to isotopy.
(If we cut $F$ off along $\partial \tilde{V}\cup\partial W$, then a  $4$-gon centered at the intersection point $\beta_{\bar{V}}\cap\alpha_{\bar{W}}$ appears and we can assume there is no more intersection point of $\partial \tilde{V}\cap\partial W$ other than the four vertices of the $4$-gon.
Hence, it is easy to see there is no bigon between $\partial \tilde{V}$ and $\partial W$ in $F$.
Hence, the bigon criterion leads to $\tilde{V}\cap W\neq \emptyset$ up to isotopy.
The symmetric argument also holds for $\tilde{W}$ and $V$.) 

This means $\tilde{V}\neq V$ and $\tilde{W}\neq W$ in $\D(F)$ by considering the weak reducing pair $(V,W)$.
Since $\tilde{V}\cap\tilde{W}=\emptyset$ up to isotopy, $\{\tilde{V},\bar{V},\bar{W},\tilde{W}\}$ forms a $3$-simplex $\Sigma_{\tilde{V}\tilde{W}}$ such that $\tilde{V}\subset\V$ and $\tilde{W}\subset \W$ and we can see $\Sigma_{VW}\cap\Sigma_{\tilde{V}\tilde{W}}=\{\bar{V},\bar{W}\}$.

This completes the proof of Claim \ref{claim-6}.
\end{proofN}

Let $V^\ast$ be an arbitrary vertex of $\mathcal{C}$ which is neither $\bar{V}$ nor $\bar{W}$.
Without loss of generality, assume $V^\ast\subset \V$.
By Lemma \ref{lemma-agt}, there is a $3$-simplex $\Sigma_{V^\ast W^\ast}$ containing $V^\ast$ in $\mathcal{C}$.
By Claim \ref{claim-6}, we can find a $3$-simplex $\Sigma_{\tilde{V}\tilde{W}}$ such that $\Sigma_{V^\ast W^\ast}\cap\Sigma_{\tilde{V}\tilde{W}}=\{\bar{V},\bar{W}\}$, i.e. $V^\ast\neq \tilde{V}$ in $\DV(F)$.
Since both $V^\ast$ and $\tilde{V}$ are band-sums of two parallel copies of $\bar{V}$ in $\V$ by Lemma \ref{lemma-agt} and $\partial V^\ast$ and $\partial \tilde{V}$ are non-isotopic in $F$, $\partial V^\ast\cap\partial \tilde{V}\neq\emptyset$ up to isotopy.
Here, $(\tilde{V},\tilde{W})$ is a weak reducing pair consisting of separating disks by Lemma \ref{lemma-agt}.
Hence, if we apply Lemma \ref{lemma-disconnected-infinite-I} to the weak reducing pair $(\tilde{V},\tilde{W})$ and the disk $V^\ast$, then we conclude the orbit $Mod(M,F).[V^\ast]$ consists of infinitely many elements, leading to the result.

This completes the proof.
\end{proof}

\section*{Acknowledgments}
This research was supported by BK21 PLUS SNU Mathematical Sciences Division.

\end{document}